\documentclass[11pt]{amsart}
\usepackage{amsmath,amssymb,latexsym,soul,cite}
\usepackage{color,enumitem,graphicx}
\usepackage[colorlinks=true,urlcolor=blue,
citecolor=red,linkcolor=blue,linktocpage,pdfpagelabels,
bookmarksnumbered,bookmarksopen]{hyperref}
\usepackage[english]{babel}
\usepackage[T1]{fontenc}
\usepackage[left=2.58cm,right=2.58cm,top=2.6cm,bottom=2.6cm]{geometry}
\usepackage[hyperpageref]{backref}

\numberwithin{equation}{section}

\newtheorem{thm}{Theorem}[section]
  \theoremstyle{plain}
  \newtheorem{lem}[thm]{Lemma}
  \theoremstyle{plain}
  
  \theoremstyle{plain}
  \newtheorem{cor}[thm]{Corollary}
  \theoremstyle{remark}
  \newtheorem{rem}[thm]{Remark}
  \theoremstyle{definition}
  \newtheorem{defi}[thm]{Definition}

\newcommand{\R}{{\mathbb R}}
\newcommand{\eps}{\varepsilon}

\newcommand{\Int}{\displaystyle \int}
\newcommand{\Frac}{\displaystyle \frac}
\newcommand{\Sum}{\displaystyle \sum}
\newcommand{\Intrn}{\displaystyle \int}
\newcommand{\Lim}{\displaystyle \lim}
\newcommand{\Limsup}{\displaystyle \limsup}
\newcommand{\Liminf}{\displaystyle \liminf}
\newcommand{\Rn}{{\mathbb{R}^n}}

\title{Asymptotically linear fractional Schr\"odinger equations}

\author[R.\ Lehrer]{Raquel Lehrer}
\address{Centro de Ciencias Exatas e Tecnologicas 
\newline\indent 
CCET, Unioeste
\newline\indent
Cascavel-PR, Brazil}
\email{rlehrer@gmail.com}

\author[L.A.\ Maia]{Liliane A. Maia}
\address{Departamento de Matematica
\newline\indent 
Universidade de Brasilia
\newline\indent
Brasilia, Brazil}
\email{lilimaia.unb@gmail.com}

\author[M.\ Squassina]{Marco Squassina}
\address{Dipartimento di Informatica
\newline\indent
Universit\`a degli Studi di Verona
\newline\indent
Verona, Italy}
\email{marco.squassina@univr.it}

\thanks{The first author is supported by CCET/UNIOESTE and the second author is partially supported by CNPq/ PQ 306388/2011-1.
The third author is partially supported by 2009 Italian MIUR project:
   ``Variational and Topological Methods in the Study of Nonlinear Phenomena''.} 
\begin{document}

\begin{abstract}
By exploiting a variational technique based upon projecting over the 
Poho\v zaev manifold, we prove existence of positive solutions for a class of nonlinear 
fractional Schr\"odinger equations having a nonhomogenous nonautonomous
asymptotically linear nonlinearity.
\end{abstract}

\maketitle

%%%%%%%%%%%%%%%%%%%%%%%%%%%%%%%%%%%%%%

\bigskip
\begin{center}
\begin{minipage}{10.5cm}
\footnotesize
\tableofcontents
\end{minipage}
\end{center}

%%%%%%%%%%%%%%%%%%%%%%%%%%%%%%%%%%%%%%
\bigskip
\smallskip

\section{Introduction and main results}

In the last few years, the study of fractional equations applied to physically relevant situations as well as to many other areas of mathematics has steadily grown.\
In \cite{metkla1,metkla2}, the authors investigate the description of anomalous diffusion
via fractional dynamics and many fractional partial differential
equations are derived from L\'evy random walk models, extending
Brownian walk models in a natural way. In particular, in \cite{laskin} a fractional Schr\"odinger
equation was obtained, which extends to a L\'evy framework a classical result that path integral over Brownian trajectories leads to
the standard Schr\"odinger equation.  More precisely,
let $s\in (0,1]$, $n>2s$ and $\mathrm{i}$ be the imaginary unit.  Then the Schr\"odinger equation
involving the fractional laplacian $(-\Delta)^s$ is 
\begin{equation}
\label{Nonlinfracev}
 \mathrm{i} \partial_t u = (-\Delta)^s u - f(x,u),   \,\,\,\quad\text{in $(0,\infty)\times\R^n$,}
\end{equation}
where the fractional Laplace operator is defined \cite{DiNezza}, for a suitable constant $C(n,s)$, as
\[
(-\Delta)^s u(x)=C(n,s)\lim_{\eps\to 0^+}\int_{\Rn\setminus B_\eps(x)}\frac{u(x)-u(y)}{|x-y|^{n+2s}}dy.
\]
Though fractional Sobolev spaces are well known since the beginning of the last century, especially among harmonic analists, 
they have become very popular in the last few year, under the impulse of the work of Caffarelli and Silvestre \cite{CS},
see again \cite{DiNezza} and the reference within.
Looking for standing wave solutions $u(t,x)=e^{\mathrm{i} \lambda t} u(x)$ of \eqref{Nonlinfracev} and assuming that
the nonlinearity is of the form $f(x,s)=a(x)f(s)$, we are led to study the following fractional equation
\begin{equation}
 (-\Delta)^{s} u + \lambda u = a(x)f(u)  \quad \text{in $\Rn$},
 \label{problema}
\end{equation}
for $\lambda >0$, whose variational formulation (weak solution) is 
\begin{equation}
\int  (-\Delta)^{s/2} u (-\Delta)^{s/2}\varphi+ \lambda \int u\varphi = \int a(x)f(u)\varphi,  
\,\,\,\quad \text{for all $\varphi\in H^s(\R^n)$}.
\label{vfor-problema}
\end{equation}
\noindent
We shall assume that $f$ satisfies the following conditions:
\begin{itemize}
\item[(f1)] $f\in C^1(\mathbb{R},\mathbb{R}^{+})$,\,\,\, $f(s)=0$\,\,\,\text{for $s\leq 0$},\,\,\, $\Lim_{s\to 0^+}\Frac{f(s)}{s} = 0$;
\vspace{0.15cm}
\item[(f2)] $\Lim_{s \to +\infty} \Frac{f(s)}{s} = 1;$
\vspace{0.15cm}
\item[(f3)] if $F(s):=\Int_{0}^{s}f(t)dt$
and $Q(s) := \Frac{1}{2}f(s)s - F(s)$, then there exists $D\geq 1$ such that
$$
Q(s)\leq D Q(t), \quad \text{for all $s \in [0,t]$},\qquad
\Lim_{s\to +\infty}Q(s) = + \infty.
$$
\end{itemize}
\noindent
On the function $a:\Rn\to\R$, we will assume the following conditions:
\begin{itemize}
\item[(A1)] $a \in C^2(\Rn,\mathbb{R}^{+})$, \quad 
$\displaystyle\inf_{\Rn}a> 0$ ;
\vspace{0.15cm}
\item[(A2)] $\Lim_{|x|\to +\infty}a(x) = a_{\infty} > \lambda$\; ;
\vspace{0.15cm}
\item[(A3)] $\nabla a(x)\cdot x \geq 0, \;\text{for all}\;  x \in \Rn$, with strict inequality on a set of  positive measure;
\vspace{0.15cm}
\item[(A4)] $a(x) + \Frac{\nabla a(x)\cdot x}{n} < a_{\infty}, \; \text{for all}\; x \in \Rn$;
\vspace{0.15cm}
\item[(A5)] $\nabla a(x)\cdot x + \Frac{x\cdot {\mathcal H}_a(x)\cdot x}{n}\geq 0, \; \text{for all}\;  x\in \Rn$, being
${\mathcal H}_a$ the Hessian matrix of $a$.
\end{itemize}

 \vskip3pt
 \noindent
Now we can state our main results. Consider 
the energy functional $I:H^s(\Rn)\to\R$,
$$
I(u) := \Frac{1}{2}\Intrn |(-\Delta)^{s/2}u|^{2} +\frac{\lambda}{2}\int u^2- \Intrn a(x) F(u),
$$
naturally associated with equation \eqref{problema}.
Then, we have the following nonexistence result

\begin{thm} \label{nonexist}
Assume that $(A1)$-$(A5)$ and $(f1)$-$(f3)$ hold and consider
$$
\mathcal{P}:= \left\lbrace u\in H^s(\Rn)\setminus\left\lbrace 0 \right\rbrace :  
\Frac{n-2s}{2}\Intrn |(-\Delta)^{s/2}u|^{2} = n\Intrn \left(\Big(a(x)+ \Frac{\nabla a(x)\cdot x}{n}\Big)F(u) - \Frac{\lambda }{2}u^2\right)\right\rbrace.
$$
Then, the infimum
\begin{equation}
\label{p-def}
\inf_{u\in \mathcal{P}}I(u),
\end{equation}
is not a critical level of $I$ and the infimum is not achieved.
\end{thm}

\noindent
Consider now also the limiting problem
\begin{equation}
\label{lim-intro}
 (-\Delta)^{s} u + \lambda u = a_\infty f(u)  \quad \text{in $\Rn$}.
\end{equation}
We shall denote by $I_\infty:H^s(\Rn)\to\R$,
$$
I_{\infty}(u) := \Frac{1}{2}\Intrn |(-\Delta)^{s/2}u|^{2} +\frac{\lambda}{2}\int u^2- \Intrn a_\infty F(u),
$$
its associated energy functional.
In Section~\ref{autonom} we shall
discuss some properties of least energy critical values of this functional.
In passing, we observe that by combining the results of \cite{lopes} (see e.g.\ Theorem 4.1 therein) with an adaptation of \cite[(i) of Lemma 1]{BJM} 
to the fractional framework, it is possible to prove that {\em any} least energy solution to \eqref{lim-intro} is radially symmetric 
and decreasing and of fixed sign.
\vskip3pt
\noindent
We have the following existence result

\begin{thm} \label{exist}
Assume that $(A1)$-$(A5)$, $(f1)$-$(f3)$ hold and
that the following facts hold
\begin{enumerate}
\item $f\in C^1(\R)\cap {\rm Lip}(\mathbb{R},\mathbb{R}^{+})$ and there exists $\tau>0$ such that
$\lim\limits_{s\to 0^+} \frac{f'(s)}{s^\tau}=0;$
\item $\|a_\infty - a\|_{L^\infty}$ is sufficiently small; 
%\noalign{\vskip3pt}
\item the least energy level $c_\infty$ of \eqref{lim-intro}
is an  isolated radial critical level for $I_\infty$ or equation \eqref{lim-intro} 
admits a unique positive solution which is radially symmetric about some point. 
\end{enumerate}
Then the nonautonomous problem 
\begin{equation*}
 (-\Delta)^{s} u + \lambda u = a(x)f(u)  \quad \text{in $\Rn$},
\end{equation*}
admits a nontrivial nonnegative solution $u \in H^s(\Rn)$.
\end{thm}

\noindent
These results extend the corresponding results in \cite{raqlili} to the fractional case.
The framework employed and ideas of the proofs of our main results follow closely those found in \cite{raqlili}.
However, the nonlocal character of the fractional laplacian requires to overcome several additional difficulties.
\vskip2pt
\noindent
%These results extend the corresponding results in \cite{raqlili} to the fractional case
%based upon the idea of projecting onto the Poho\v zaev manifold. The nonlocal character
%of the fractional laplacian requires to overcome several additional difficulties.
Theorem~\ref{exist} follows 
under uniqueness of positive radial solutions of \eqref{lim-intro} or isolatedness
assumption on the least energy level of $I_\infty$. 
To our knowledge, in the case $s\in (0,1)$,
the isolatedness or uniqueness assumption of Theorem~\ref{exist} are unknown
in the current literature. In the case $s=1$, it follows for instance by the uniqueness
result by Serrin-Tang \cite{serrin}, under suitable assumptions of $f(s)$ for large values of $s$,
which are compatible with the model nonlinearity 
\begin{equation}
\label{model-nonlin}
f(s)=\frac{s^3}{1+s^2},\quad \text{for $s\geq 0$},\qquad f(s)=0,\quad\text{for $s\leq 0$}.
\end{equation}
In fact, assumptions (H1)-(H2) in \cite[Theorem 1]{serrin} are fulfilled with 
$b=(\lambda/(a_\infty-\lambda))^{1/2}>0,$ where $a_\infty>\lambda.$
For the case of superquadratic nonlinearities $f(s)=s^p$, nondegeneracy and uniqueness properties of ground state solutions of \eqref{lim-intro}
where recently proved in \cite{frank,frank2}, so assumption (3) of Theorem~\ref{exist}
is expected to be fulfilled.
Semi-linear Schr\"odinger equations associated with the asymptotically linear model 
nonlinearity \eqref{model-nonlin} are one of the main motivations 
for developing the technique in this paper.
For the physical background in the local case $s=1$, see \cite{stuart1,stuart2}.
\vskip2pt
\noindent
In \cite{chang}, the author considers asymptotically linear 
fractional NLS with an external potential $V$ which provides compactness directly
via coercivity.
We also refer the reader to the contributions \cite{ChangWang,felmer} where the case of a 
superquadratic nonlinearity is covered for the fractional laplacian obtaining
existence, regularity and qualitative properties of solutions. In the superquadratic case, as known,
one can also exploit the Nehari manifold associated with the problem. On the other hand, when
the nonlinear term is nonhomogeneous and asymptotically linear, as it was pointed out by Costa and Tehrani
in\cite{costa}, in general, not every nonzero function can be projected onto the Nehari manifold 
or it may happen that the projection is not uniquely determined. In turn, as exploitied in other 
contributions \cite{azz,jeanta2,raqlili}, we shall look at projections onto the Poho\v zaev manifold in place of the Nehari constraint in order
to prove Theorem~\ref{nonexist} and \ref{exist}.
\vskip6pt
\noindent
{\bf A few additional remarks.}
Conditions $(A2)$, $(A3)$ and $(A4)$ imply
\begin{equation} \label{antigaA5}
\nabla a(x)\cdot x \rightarrow 0, \quad \text{if } \quad|x|\to+\infty,
\end{equation}
while $(f1)$ and $(f2)$ imply that, given $\varepsilon > 0$ and $2<p\leq 2n/(n-2s)$, 
there exists $C_\varepsilon>0$ with
\begin{equation}
|F(s)|\leq \Frac{\varepsilon}{2}|s|^2 + C_\eps |s|^p,\qquad\text{for all $s\in \R$.}
\label{condiF2}
\end{equation}
In what follows we will denote 
\begin{equation}
\label{Hsnorm}
\Vert u \Vert_{H^s}= \Big(\Intrn |(-\Delta)^{s/2}u|^{2}+ \Intrn\lambda u^2\Big)^{1/2},
\end{equation} 
as the norm in $H^s(\Rn)$, which is equivalent to the standard norm of $H^{s}(\Rn)$.
We will also denote by $\Vert u \Vert_{p}$ de usual norm of $L^{p}(\Rn)$.
We define $I: H^{s}(\Rn) \rightarrow \mathbb{R}$ as the functional associated with \eqref{problema}
$$
I(u):= \Frac{1}{2}\Intrn |(-\Delta)^{s/2}u|^{2}- \Intrn G(x,u),
\qquad
G(x,u) :=a(x)F(u) - \frac{\lambda}{2} u^2.
$$
Since $f(s)=0$ on $\R^-$, it follows that any weak solution $u\in H^s(\R^n)$ for \eqref{problema}
is nonnegative. In fact, by choosing $\varphi=u^-\in H^s(\Rn)$ in the variational formulation \eqref{vfor-problema} yields
\begin{equation*}
\int  (-\Delta)^{s/2} u (-\Delta)^{s/2} u^- = \int a(x)f(u)u^-- \lambda \int u u^-= \lambda \int (u^-)^2.
\end{equation*}
Hence, if $C(n,s)$ is the normalization constant in the definition of $(-\Delta)^s$, we obtain
\begin{align}
\label{conput-pos}
& \int  (-\Delta)^{s/2} u (-\Delta)^{s/2} u^-  =\int  u^- (-\Delta)^s u^+ -\|(-\Delta)^{s/2} u^-\|_2^2 \\ 
& =\frac{C(n,s)}{2}\iint  \frac{(u^+(x)-u^+(y))(u^-(x)-u^-(y))}{|x-y|^{n+2s}}dxdy-\|(-\Delta)^{s/2} u^-\|_2^2   \notag \\
& =-C(n,s)\iint  \frac{u^+(x)u^-(y)}{|x-y|^{n+2s}}dxdy-\|(-\Delta)^{s/2} u^-\|_2^2\leq -\|(-\Delta)^{s/2} u^-\|_2^2.  \notag
\end{align}
In turn we get $\|u^-\|_{H^s}^2=\|u^-\|_2^2+\|(-\Delta)^{s/2} u^-\|_2^2=0$, 
namely $u^-=0$, hence the assertion.

\section{Energy levels of the limiting problem}
\label{autonom}
\noindent
In this section we study the following equation for $s\in (0,1)$ and $n>2s$,
\begin{equation}
(-\Delta)^{s} u + \lambda u = a_{\infty}f(u),\qquad \text{in $\R^n$},
\label{auto}
\end{equation}
where $\lambda>0$ and $a_{\infty} >\lambda$. 
We shall assume that $F$ satisfies the growth estimate \eqref{condiF2}.
Our aim is to provide a Mountain Pass 
characterization for least energy solutions which is the counterpart of the main result
of \cite{jeanta2}. 
Let the Hilbert space $H^s(\R^n)$ be endowed with the norm \eqref{Hsnorm}
and let $I_{\infty}: H^{s}(\Rn) \to\R$ be the functional corresponding to \eqref{auto}, namely
$$
I_{\infty}(u) = \Frac{1}{2}\Intrn |(-\Delta)^{s/2}u|^{2} - \Intrn G_{\infty}(u),
$$
where we have set
$$
G_{\infty}(u) =\displaystyle\int_{0}^{u} g_\infty(t)dt=
\displaystyle\int_{0}^{u} (a_{\infty}f(t) - \lambda t )dt=a_\infty F(u)-
\frac{\lambda}{2}u^2.
$$
We say that a solution $u$ of $\eqref{auto}$ is 
a \textit{least energy solution} to \eqref{auto} if 
$$
I_{\infty}(u) = m,\qquad
m := \inf\big\{I_{\infty}(u): \text{$u\in H^{s}(\Rn)\setminus \left\lbrace 0 \right\rbrace$ is a solution of $\eqref{auto}$}\big\}.
$$
As stated in \cite[Theorem 1.1 ]{R-OS}, the Poho\v zaev identity associated with 
\eqref{auto} is given by 
\begin{equation*}
(n-2s)\Intrn ug_\infty(u)  = 2n \Intrn G_{\infty}(u),
\end{equation*}
where $g_\infty$ and $G_{\infty}$ are defined as before.
Also, if $u$ and $v$ belong to $H^{s}(\Rn)$, then
$$
\Intrn v(-\Delta)^{s} u  = \Intrn (-\Delta)^{s/2}u (-\Delta)^{s/2}v ,
$$
which yields in turn
$$
\Intrn ug_\infty(u) = \Intrn |(-\Delta)^{s/2}u|^{2}=\Frac{C(n,s)}{2}[u]^{2}_{H^{s}},
\qquad
[u]_{H^{s}}:=\Big(\iint \Frac{|u(x)- u(y)|^{2}}{|x-y|^{n+2s}}dxdy\Big)^{1/2}.
$$
Therefore, the Poho\v zaev identity may be written 
as (see also \cite[Proposition 4.1]{ChangWang} for a different proof)
\begin{equation}
(n-2s)\Intrn |(-\Delta)^{s/2}u|^{2} = 2n\Intrn G_{\infty}(u).
\label{pohozaevinf}
\end{equation}
For the following, it is convenient to introduce the set 
$$
\mathcal{P}_{\infty}:= \big\{u\in H^{s}(\Rn)\setminus \left\lbrace 0 \right\rbrace: \text{$u$  satisfies identity \eqref{pohozaevinf}}\big\}.
$$
We also consider we set of paths
$$
\Gamma_{\infty}:=\big\{\gamma\in C([0,1],H^s(\Rn)):\, \gamma(0)=0,\,\, I_{\infty}(\gamma(1))<0\big\},
$$ 
and define the min-max Mountain Pass level (see \cite{ambrorabino})
\begin{equation}
\label{cinftydef}
c_{\infty} := \displaystyle\min_{\gamma\in \Gamma_{\infty}}\max_{t\in[0,1]}I_{\infty}(\gamma(t)).
\end{equation}

\noindent
The main result of the section is the following
\begin{thm} \label{cim}
$c_{\infty}=m$.
\end{thm}

\noindent
In order to prove the result we need the following Lemmas.

\begin{lem}
\label{first-lem}
Let $w \in H^{s}(\Rn)$ be a least energy solution to \eqref{auto}. Then
there exists $\gamma \in \Gamma_{\infty}$ such that 
$$
w\in \gamma([0,1]),\qquad 
\max_{t\in[0,1]}I_{\infty}(\gamma(t)) = I_{\infty}(w) = m.
$$
\end{lem}

\begin{proof}
Consider a least energy solution $w$ of \eqref{auto}, which exists e.g.\ by 
 \cite[Theorem 1.1]{ChangWang}. 
Then we can define the continuous path $\alpha:[0,\infty)\to H^s(\R^n)$ by setting
$\alpha(t)(x):=w(x/t)$, if $t>0$, and $\alpha(0):=0$.
Then, by construction, we have $I_{\infty}(\alpha(0)) = 0$ and 
\begin{align*}
I_{\infty}(\alpha(t))& =  \Frac{1}{2}\Intrn |(-\Delta)^{s/2} w(x/t)|^{2} - \Intrn G_{\infty}(w(x/t))\\
& = \Frac{t^{n-2s}}{2}\Intrn |(-\Delta)^{s/2}w(x)|^{2} - t^{n}\Intrn G_{\infty}(w),\quad t>0.
\end{align*}
Then, taking the derivative, we obtain
\begin{align*}
\Frac{d}{dt}I_{\infty}(\alpha(t)) & = \Frac{(n-2s)}{2}t^{n-2s-1}\Intrn |(-\Delta)^{s/2} w|^{2} - nt^{n-1}\Intrn G_{\infty}(w)\\
& = \Frac{t^{n-2s-1}}{2}\left\lbrace (n-2s)\Intrn |(-\Delta)^{s/2}w|^{2}  - 2n t^{2s}\Intrn G_{\infty}(w)\right\rbrace
\end{align*}
Since $w$ is a solution of $(\ref{auto})$, it satisfies the Poho\v zaev identity \eqref{pohozaevinf}, therefore
\begin{equation*}
\Frac{d}{dt}I_{\infty}(\alpha(t)) = \Frac{t^{n-2s-1}}{2}(n-2s)(1-t^{2s})\Intrn |(-\Delta)^{s/2}w|^{2}.
\end{equation*}
Then, since $n>2s$, the map $\{t\mapsto I_{\infty}(\alpha(t))\}$ achieves
the maximum value at $t=1$. By choosing $L>0$ sufficiently large and recalling \eqref{pohozaevinf} again
to guarantee $\int G_{\infty}(w)>0$, we have 
$$
\max_{0\leq t \leq L}I_{\infty}(\alpha(t)) = I_{\infty}(\alpha(1)) =  I_{\infty}(w) =m,\qquad I_{\infty}(\alpha(L)) <0.
$$ 
Taking $\gamma(t) := \alpha(tL)$, we have that $\gamma \in \Gamma_{\infty}$ and the result follows.
\end{proof}

\begin{lem}
\label{pre-lem1}
$\gamma([0,1])\cap \mathcal{P}_{\infty} \neq \emptyset$, for all $\gamma \in \Gamma_{\infty}$.
\end{lem}

\begin{proof}
Consider the functional associated with the Poho\v zaev identity \eqref{pohozaevinf},
\begin{equation}
\label{funct-pohoz}
J_{\infty}(u):= \Frac{n-2s}{2}\Intrn |(-\Delta)^{s/2}u|^{2} - n\Intrn G_{\infty}(u),
\qquad u\in H^s(\Rn).
\end{equation} 
We will first prove 
that there exists $\rho >0$ such that, if $0<\Vert u \Vert_{H^s}\leq \rho$, then $J(u)>0$. We have
\begin{align*} 
J_{\infty}(u) &= \Frac{n-2s}{2}\Intrn |(-\Delta)^{s/2}u|^{2} - na_{\infty}\Intrn F(u) + \Frac{n\lambda}{2}\Intrn u^2  \\
&\geq \Frac{n-2s}{2}\Vert u \Vert^2_{H^s} - na_{\infty} \Intrn F(u).
\end{align*}
Then, by virtue of \eqref{condiF2} and the fractional Sobolev inequality
\cite[Theorem 6.7]{DiNezza}, we obtain
\begin{align*}
J_{\infty}(u)&\geq    \Frac{n-2s}{2}\Vert u \Vert^2_{H^s} 
- \Frac{n\varepsilon a_{\infty}}{2\lambda}\Intrn  \lambda u^2 - na_{\infty} C_\eps \Intrn |u|^p \\
& \geq  \Frac{1}{2}\left(n-2s- \frac{n\varepsilon a_{\infty}}{\lambda}\right)\Vert u \Vert^2_{H^s}  - na_{\infty}C_\eps \Vert u \Vert^p_{H^s} .
\end{align*}
Take now $\varepsilon>0$ so small that $n-2s- 
n\varepsilon a_{\infty}/\lambda >0$ and then choose
$\rho>0$ small enough so that $J_{\infty}(u) >0$ if $0< \Vert u \Vert_{H^s}\leq \rho$,  which is possible, since $p>2$. Observe now that
\begin{equation*}
J_{\infty}(u) = n I_{\infty}(u) -s\Intrn |(-\Delta)^{s/2}u|^{2}.
\end{equation*}
If $\gamma \in \Gamma_{\infty}$, we have $J_{\infty}(\gamma(0))=0$ and $J_{\infty}(\gamma(1)) \leq nI_{\infty}(\gamma(1)) <0$.
Then, by continuity, there exists $\sigma\in(0,1)$ such that $\Vert \gamma(\sigma)\Vert_{H^s} \geq  \rho$ and $J_{\infty}(\gamma(\sigma)) = 0$.  
This means $\gamma(\sigma)\in \mathcal{P}_{\infty}$, concluding the proof.
\end{proof}

\begin{lem}
\label{pre-lem2}
We have
$$
m=\inf_{u\in {\mathcal P}_\infty} I_\infty(u).
$$
\end{lem}
\begin{proof}
If we set
$$
{\mathcal S}_\infty=\Big\{u\in H^s(\R^n):\int G_\infty(u)=1\Big\},
$$
it follows that $\Phi:{\mathcal S}_\infty\to{\mathcal P}_\infty$ defined by
$$
\Phi(u)(x):=u\Big(\frac{x}{t_u}\Big),\qquad t_u:=\Big(\frac{n-2s}{2n}\Big)^{1/2s}\|(-\Delta)^{s/2} u\|_2^{1/s} 
$$
establishes a bijective correspondence and
$$
I_\infty(\Phi(u))=\frac{s}{n}\Big(\frac{n-2s}{2n}\Big)^{(n-2s)/2s}\|(-\Delta)^{s/2} u\|_2^{n/s},
\,\,\quad u\in {\mathcal S}_\infty,
$$ 
yielding in turn
$$
\inf_{u\in {\mathcal P}_\infty} I_\infty(u)=\inf_{u\in {\mathcal S}_\infty} I_\infty(\Phi(u))=
\inf_{u\in {\mathcal S}_\infty} \frac{s}{n}\Big(\frac{n-2s}{2n}\Big)^{(n-2s)/2s}\|(-\Delta)^{s/2} u\|_2^{n/s}=m, 
$$
since the last infimum is achieved and the corresponding value 
equals the least energy level $m$. This can be proved by performing calculations
similar to those of \cite[proof of (i) of Lemma 1]{BJM}. 
\end{proof}

\noindent
{\em Proof of Theorem~\ref{cim} concluded.\ }
By combining Lemma~\ref{pre-lem1} with \ref{pre-lem2}
we immediately obtain $m\leq c_{\infty}$. Considering the path 
$\gamma \in \Gamma_{\infty}$ provided by Lemma~\ref{pre-lem1}, we have
$$
\max_{0\leq t \leq 1}I_{\infty}(\gamma(t)) = I_{\infty}(w)=m.
$$
By taking the infimum over $\Gamma_{\infty}$ yields
$$
\inf_{\gamma \in \Gamma_{\infty}}\max_{t \in [0,1]}I_{\infty}(\gamma(t)) \leq m,
$$
so that $c_{\infty}\leq m$, which concludes the proof.
\qed

\section{Projecting on the Poho\v zaev manifold}

\noindent
From \cite[Proposition 1.12]{R-OS}, if $u\in H^{s}(\Rn)$ is a 
solution of $(\ref{problema})$, then $u$ satisfies de Poho\v zaev identity
\begin{equation}
\Frac{n-2s}{2}\Intrn |(-\Delta)^{s/2}u|^{2} = n\Intrn \left(\Big(a(x)+ \Frac{\nabla a(x)\cdot x}{n}\Big)F(u) - \Frac{\lambda }{2}u^2\right). 
\label{pohozaev}
\end{equation}
Furthermore, we define the Poho\v zaev manifold associated with \eqref{problema} by
$$
\mathcal{P}:= \left\lbrace u\in H^s(\Rn)\setminus\left\lbrace 0 \right\rbrace : u \ \text{satisfies identity $\eqref{pohozaev}$}\right\rbrace.
$$

\noindent
We first have the following
\begin{lem} \label{manifold}
Let the functional $J:H^s(\Rn)\rightarrow \mathbb{R}$ be defined by
$$
J(u) := \Frac{n-2s}{2}\Intrn |(-\Delta)^{s/2}u|^{2} - n\Intrn \left(\Big(a(x)+ \Frac{\nabla a(x)\cdot x}{n}\Big)F(u) - \Frac{\lambda u^{2}}{2}\right).
$$
Then, it holds that
\begin{itemize}
\item[a)] $\left\lbrace u \equiv 0 \right\rbrace$ is an isolated point of $J^{-1}(\left\lbrace 0 \right\rbrace)$;
\item[b)] $\mathcal{P}:= \left\lbrace u\in H^s(\Rn) \setminus \left\lbrace 0 \right\rbrace : J(u)=0\right\rbrace$ is a closed set.
\item[c)] $\mathcal{P}$ is a $C^1$ manifold.
\item[d)] There exists $\sigma >0$ such that $\Vert u \Vert_{H^s} > \sigma,$ for all $ u \in \mathcal{P}$.
\end{itemize}
\end{lem}

\begin{proof} $(a)$
Using condition $(A4)$, we get
\begin{align*}
J(u) & = \Frac{n-2s}{2}\Intrn |(-\Delta)^{s/2}u|^{2}  - n\Intrn \left(\Big(a(x)+ \Frac{\nabla a(x)\cdot x}{n}\Big)F(u) - \Frac{\lambda u^{2}}{2}\right)\\
&>  \Frac{n-2s}{2}\Intrn |(-\Delta)^{s/2}u|^{2} - n\Intrn \left(a_{\infty}F(u) - \lambda\Frac{u^2}{2}\right) \\
&\geq \Frac{n-2s}{2}\Vert u \Vert^2_{H^s} - na_{\infty}\Intrn F(u).
\end{align*}
By virtue of the fractional Sobolev embedding \cite[Theorem 6.7]{DiNezza} and 
\eqref{condiF2}, we obtain
\begin{align*}
J(u)&\geq  \Frac{n-2s}{2}\Vert u \Vert^2_{H^s}- \Frac{n\varepsilon a_{\infty}}{2\lambda}\Intrn \lambda u^2 - na_{\infty}C_\eps\Intrn |u|^p \\
& \geq  \Frac{1}{2}\left( n - 2s - \Frac{n\varepsilon a_{\infty}}{\lambda}\right)\Vert u \Vert^2_{H^s} - na_{\infty}C_\eps \Vert u \Vert^p_{H^s}.
\end{align*}
Take $\varepsilon>0$ with
$n-2s - n\varepsilon a_{\infty}/\lambda >0$.  
For $\rho>0$ small enough, 
$J(u) >0$ if $0< \Vert u \Vert_{H^s} < \rho$.
\vskip4pt
\noindent
$(b)$ $J(u)$ is a $C^1$ functional, thus $\mathcal{P} \cup \left\lbrace 0\right\rbrace = J^{-1}(\left\lbrace 0\right\rbrace)$ is a closed subset. Moreover, $\left\lbrace u \equiv 0 \right\rbrace$ is an isolated point in $J^{-1}(\left\lbrace 0 \right\rbrace)$ and 
the assertion follows.
\vskip4pt
\noindent
$(c)$ Considering the derivative of $J$ at $u$ and applied at $u$ yields
\begin{equation}
J'(u)(u) = (n-2s)\Intrn |(-\Delta)^{s/2}u|^{2} - n\Intrn \left(a(x)f(u)u - \lambda u^2\right) - \Intrn \nabla a(x)\cdot xf(u)u. \label{j(u)u}
\end{equation}
Since $u\in \mathcal{P}$, it follows that $u$ satisfies $(\ref{pohozaev})$
and, using formula \eqref{pohozaev} into \eqref{j(u)u}, we obtain
\begin{align*}
J'(u)(u) & =  2n\Intrn a(x)F(u) - n\lambda\Intrn u^2 + 2\Intrn \nabla a(x)\cdot xF(u) \\
&- n\Intrn a(x)f(u)u+n\lambda \Intrn u^2 - \Intrn \nabla a(x)\cdot xf(u)u \\
&= 2n\Intrn \Big( a(x)+ \Frac{\nabla a(x)\cdot x}{n}\Big)F(u)  \\
&- n\Intrn \Big(a(x)+\Frac{\nabla a(x)\cdot x}{n}\Big)f(u)u \\
&= 2n\Intrn \Big(a(x)+ \Frac{\nabla a(x)\cdot x}{n}\Big)\Big(F(u)-\Frac{1}{2}f(u)u\Big)  <0,
\end{align*}
in light of $(A1)$, $(A3)$ and $(f3)$. Then, 
if $ u\in \mathcal{P}$, then  $J'(u)(u)<0$. 
This shows that the set $\mathcal{P}$ is a $C^1$ manifold.
\vskip4pt
\noindent
$(d)$ Since $0$ is isolated in $J^{-1}(\{0\})$, 
there is a ball $\|u\|_{H^s} \leq \sigma$ which does not intersect $\mathcal{P}$.
\end{proof}

\section{Nonexistence results}

\noindent
In this section we
get relations between the Poho\v zaev manifold  $\mathcal{P}$ associated with 
\eqref{problema} and the Poho\v zaev manifold $\mathcal{P}_{\infty}$ 
for the limiting problem \eqref{auto}. Recall that
$$
\mathcal{P}_{\infty} = \left\lbrace u\in H^s(\Rn)\setminus 
\left\lbrace 0\right\rbrace : J_{\infty}(u) = 0\right\rbrace,
$$ 
where $J_\infty$ is defined as in \eqref{funct-pohoz}.
%$$
%J_{\infty}(u) := \Frac{(n-2s)}{2}\Intrn |(-\Delta)^{s/2}u|^{2} - n\Intrn G_{\infty}(u),
%\qquad
%G_{\infty}(u) :=a_{\infty}F(u) - \lambda\Frac{u^2}{2}.
%$$
Notice that the hypotheses $(A3)$-$(A4)$ 
imply that $I_\infty(u) < I(u)$ for every $u$ in $H^s(\Rn) \setminus \{0\}$.
If $p$ is defined as in \eqref{p-def}, we will show in this section that
$p=c_{\infty},$ that this level is not critical for $I$ and in turn
that it is not achieved.

\begin{lem}\label{pohoprojeta}
If $\int G_{\infty}(u)>0$, there exist unique $\vartheta_1,\vartheta_2 >0$ with
$u(\cdot/\vartheta_1)\in \mathcal{P}_{\infty}$ and $u(\cdot/\vartheta_2)\in \mathcal{P}$.
\end{lem}

\begin{proof}
First, we consider the case of $\mathcal{P}_{\infty}$. Consider the function 
$\varphi:(0,\infty)\to\R$ defined by
\begin{equation*}
\varphi(\vartheta)  :=  I_{\infty}(u(x/\vartheta)) 
= \Frac{\vartheta^{n-2s}}{2}\Intrn |(-\Delta)^{s/2}u|^{2}  - \vartheta^{n}\Intrn G_{\infty}(u).
\end{equation*}
Taking the derivative of $\varphi$, we obtain
\begin{align*}
\varphi'(\vartheta)  &=\Frac{\vartheta^{n-2s-1}}{2}\Big((n-2s)\Intrn |(-\Delta)^{s/2}u|^{2} - 2n\vartheta^{2s}\Intrn G_{\infty}(u) \Big) \\
&=\Frac{1}{\vartheta}\Big(\frac{n-2s}{2}\Intrn |(-\Delta)^{s/2} u(x/\vartheta)|^{2} 
- n\Intrn G_{\infty}(u(x/\vartheta)) \Big)=
\frac{J_\infty(u(\cdot/\vartheta))}{\vartheta}.
\end{align*}
Then, $\varphi '(\vartheta) = 0$ if and only if either $\vartheta = 0 $ or
$$
\vartheta=\vartheta_1= \Big(\frac{n-2s}{2n}\Frac{\int |(-\Delta)^{s/2}u|^{2}}{\int G_{\infty}(u)}\Big)^{1/2s} >0.
$$
Since by the formula for $\varphi'$ we have
$u(x/\vartheta)\in \mathcal{P}_{\infty}$ if and only if $\varphi'(\vartheta)=0$ for some $\vartheta>0$, we have the result.
In passing, we observe that $\varphi$ is positive for $\vartheta>0$ small while it is negative for $\vartheta>0$ large, so
that the unique critical point of $\varphi$ corresponds to a global maximum point for $\varphi$.
Now we turn to the case of $\mathcal{P}$. First, we define the function $\Psi:(0,\infty)\to\R$
by
\begin{align}
\label{basic}
\Psi(\vartheta)&:=  I(u(x/\vartheta)) = \Frac{\vartheta^{n-2s}}{2}\Intrn |(-\Delta)^{s/2}u|^{2}  - \Intrn G(x,u(x/\vartheta))  \\
& = \Frac{\vartheta^{n-2s}}{2}\Intrn |(-\Delta)^{s/2}u|^{2}  - \Intrn \Big(a(x)F(u(x/\vartheta)) - \lambda\Frac{u^2(x/\vartheta)}{2}\Big)  \notag \\
&=\Frac{\vartheta^{n-2s}}{2}\Intrn |(-\Delta)^{s/2}u|^{2}  - \vartheta^n\Intrn \Big( a(\vartheta x)F(u) - \lambda\Frac{u^2}{2}\Big).  \notag
\end{align}
Taking the derivative of $\Psi$ and recalling that $n>2s$, we obtain:
\begin{align}
\label{formPsi}
\Psi'(\vartheta)& = \Frac{n-2s}{2}\vartheta^{n-2s-1}\Intrn |(-\Delta)^{s/2}u|^{2} - n\vartheta^{n-1}\Intrn \Big(a(\vartheta x)F(u) - \lambda\Frac{u^2}{2}\Big)  \\
&- \vartheta^n\Intrn\nabla a(\vartheta x)\cdot x F(u)  \notag \\
& =  \vartheta^{n-2s-1}\left\lbrace \Frac{n-2s}{2}\Intrn |(-\Delta)^{s/2}u|^{2} - n\vartheta^{2s}\Intrn \Big(a(\vartheta x)F(u) - \lambda\Frac{u^2}{2}\Big) \right.  \notag \\
&- \left. \vartheta^{2s} \Intrn \nabla a(\vartheta x)\cdot (\vartheta x)F(u) \right\rbrace \notag \\
& =  \frac{1}{\vartheta}\left\lbrace \Frac{n-2s}{2}\Intrn |(-\Delta)^{s/2}u(x/\vartheta)|^{2} - n\Intrn \Big(a(x)F(u(x/\vartheta)) - \lambda\Frac{u^2(x/\vartheta)}{2}\Big) \right.  \notag \\
&- \left. \Intrn \nabla a(x)\cdot x F(u(x/\vartheta)) \right\rbrace=
\frac{J(u(\cdot/\vartheta))}{\vartheta}.   \notag
\end{align}
Hence, $u(\cdot/\vartheta) \in \mathcal{P}$ if 
and only if $\Psi'(\vartheta) =0$, for some $\vartheta >0$. 
Notice that, in view of $(A2)$ and \eqref{antigaA5} and the 
Lebesgue Dominated Convergence Theorem, we get
\begin{align*}
& \Lim_{\vartheta\to \infty} \Intrn a(\vartheta x)F(u) - \lambda\Frac{u^2}{2} 
=\Intrn G_{\infty}(u), \\
& \Lim_{\vartheta\to \infty}\Intrn \nabla a(\vartheta x)\cdot (\vartheta x)F(u) = 0.
\end{align*}
Therefore, if $\vartheta >0$ is sufficiently large, then 
$$
\Psi'(\vartheta) = \vartheta^{n-2s-1}\left\lbrace \Frac{n-2s}{2}\Intrn |(-\Delta)^{s/2}u|^{2} - n\vartheta^{2s}\Big(\Intrn G_{\infty}(u) + o_{\vartheta}(1)\Big)\right\rbrace.
$$
Since $\int G_{\infty}(u) >0$, it follows that 
$\Psi'(\vartheta) < 0$, for $\vartheta >0$ sufficiently large.
On the other hand, if $\vartheta >0$ is sufficiently small we have that condition $(A4)$, together with $(A1)$-$(A3)$ yield
$$
0 < a(x)+ \Frac{\nabla a(x)\cdot x}{n} < a_{\infty},
$$
$$
-\Frac{\lambda}{2}\Intrn u^2 \leq \Intrn \left(\Big(a(\vartheta x) + \Frac{\nabla a(\vartheta x)\cdot (\vartheta x)}{n}\Big)F(u)-\lambda\Frac{u^2}{2}\right) < \Intrn G_{\infty}(u)
\leq \Frac{a_{\infty} C}{2}\Intrn u^2,
$$
where $C$ is a positive constant independent of $\vartheta$.
Thus, taking  $\vartheta>0$ sufficiently small in $\Psi'(\vartheta)$, we obtain $\Psi'(\vartheta) >0$. 
Since $\Psi'$ is continuous, there exists $\vartheta_2=\vartheta_2(u)>0$ such that 
$\Psi'(\vartheta_2) = 0$, which means that $u(\cdot/\vartheta_2)\in \mathcal{P}$.
To show the uniqueness of $\vartheta_2$, note that $\Psi'(\vartheta)=0$ implies
\begin{equation}
\label{charatcht}
\Frac{n-2s}{2}\Intrn |(-\Delta)^{s/2}u|^{2} = n\vartheta^{2s} h(\vartheta),\qquad
h(\vartheta) :=  \Intrn \Big(\Big( a(\vartheta x)+ \Frac{\nabla a(\vartheta x)\cdot(\vartheta x)}{n}\Big)F(u) - \Frac{\lambda u^2}{2}\Big) ,
\end{equation}
with $\vartheta> 0$. Taking the derivative of $h$ we end up with
$$
h'(\vartheta) = \Frac{1}{\vartheta} \Intrn \left( \nabla a(\vartheta x)\cdot(\vartheta x) +
\Frac{(\vartheta x)\cdot {\mathcal H}_a(\vartheta x)\cdot (\vartheta x)}{n} + \Frac{\nabla a(\vartheta x)\cdot (\vartheta x)}{n}\right)F(u).
$$
Hypotheses $(A3)$ and $(A5)$ imply that $h'(\vartheta)>0$.
Therefore,  $h$ is an increasing function of $\vartheta$ and hence 
there exists a unique $\vartheta >0$ such that the identity in \eqref{charatcht} holds. 
As for the functional $\varphi$, the above arguments show that $\Psi$ is positive for $\vartheta>0$ small while 
it is negative for $\vartheta>0$ large, and hence
the unique critical point of $\Psi$ corresponds to a global maximum point for $\Psi$.
\end{proof}

%\begin{rem}
%Hypothesis $(A5)$ was used in the previous lemma only to show the uniqueness of $\theta_2$.
%\end{rem}

\noindent
Consider the open subset of $H^s(\R^n)$
$$
\mathcal{O} = \Big\lbrace u\in H^s(\Rn)\setminus \left\lbrace 0 \right\rbrace:\, \int G_{\infty}(u)>0 \Big\rbrace.
$$ 
Then we have the following

\begin{lem}
The map defined by $ \mathcal{O}\ni u\mapsto \theta_2(u)\in (0,\infty)$, such that $u(\cdot/ \theta_2(u))\in \mathcal{P}$, is continuous.
\end{lem}

\begin{proof}
Let $u\in \mathcal{O}$ and consider  $(u_j)\subset \mathcal{O}$ 
such that $u_j \to u$ in $H^s(\R^n)$ as $j\to\infty$. 
First note that $\vartheta_2(u_j)$ is bounded.
Indeed, consider the expression \eqref{charatcht} of $\psi'= 0$ in the proof 
of Lemma~\ref{pohoprojeta} for $u_j$ and $\vartheta_2(u_j)$
\begin{equation*}
\Frac{n-2s}{2}\Intrn |(-\Delta)^{s/2}u_j|^{2} =  n\vartheta^{2s}_2(u_j)\Intrn \Big(\Big(a(\vartheta_2(u_j) x)  
+ \Frac{\nabla a(\vartheta_2(u_j) x)\cdot(\vartheta_2(u_j) x)}{n}\Big)F(u_j)   
- \Frac{\lambda u^2_j}{2}\Big). 
\end{equation*} 
Suppose by contradiction that $\vartheta_2(u_j)\to\infty$ as $j\to\infty$, 
along a suitable subsequence. 
Then, in light of the assumptions on $a$ and $F$ and by 
Lebesgue Dominated Convergence Theorem,
the right-hand side of the above equation goes to 
infinity while the left-hand converges to
$(n-2s)/2\|(-\Delta)^{s/2}u\|^{2}_2,$
which is a contradiction. Hence, $\vartheta_2(u_j)$ admits a convergent 
subsequence, say $\vartheta_2(u_j)\rightarrow \bar\vartheta$ as $j\to\infty$.
In turn, by Lebesgue Dominated Convergence Theorem, as $j\to\infty$, we have
\begin{align*}
\Intrn a(\vartheta_2(u_j)x)F(u_j)  & \to \Intrn a(\bar{\vartheta}x)F(u),   \\
\Intrn \nabla a(\vartheta_2(u_j) x)\cdot(\vartheta_2(u_j) x) F(u_j) & \to \Intrn \nabla a(\bar{\vartheta}x)\cdot (\bar{\vartheta}x)F(u).
\end{align*}
Then, since $u_j\rightarrow u$ in $H^s(\Rn)$ as $j\to\infty$, we obtain
$$
\Frac{n-2s}{2}\Intrn |(-\Delta)^{s/2}u|^{2} = n\bar{\vartheta}^{2s} 
\Intrn \Big( \Big(a(\bar{\vartheta}x) + \Frac{\nabla a(\bar{\vartheta}x)\cdot (\bar{\vartheta}x)}{n}\Big)F(u) - \Frac{\lambda u^2}{2}\Big).
$$
Hence $u(\cdot/\bar{\vartheta})\in \mathcal{P}$ and, 
by uniqueness of the projection in
$\mathcal{P}$, $\bar{\vartheta}= \vartheta_2(u)$.
\end{proof}

\begin{lem}\label{mini}
If $u \in \mathcal{P}_{\infty}$, then $\int G_{\infty}(u) >0$.
\end{lem}
\begin{proof}
Let $u \in \mathcal{P}_{\infty}.$ Of course $\int G_{\infty}(u) \geq 0$. Assume by contradiction that $\int G_{\infty}(u)=0.$ Then
$$
0=\|(-\Delta)^{s/2} u\|_2^2=\frac{C(n,s)}{2}\iint  \frac{(u(x)-u(y))^2}{|x-y|^{n+2s}}dxdy,
$$
so that $u$ is constant and hence, as $u\in L^2(\R^n)$, $u=0$, contradicting $u \in \mathcal{P}_{\infty}$.
\end{proof}

\begin{lem}\label{pohoprojeta2}
If $u \in \mathcal{P}_{\infty}$, then there exists a unique $\vartheta > 0$ such that $u(\cdot/\vartheta)\in \mathcal{P}$ and $\vartheta >1$.
\end{lem}

\begin{proof}
Let $u \in \mathcal{P}_{\infty}$. Then, by Lemma~\ref{mini}, $\int G_{\infty}(u) >0$. 
In turn, by Lemma~\ref{pohoprojeta},
there exists a unique $\vartheta>0$ such that $u(\cdot/\vartheta)\in \mathcal{P}$. Now, we are left with the proof that $\vartheta>1$. By the arguments in the previous lemmas, it follows that $\vartheta$ satisfies
$$
\Frac{n-2s}{2}\Intrn |(-\Delta)^{s/2}u|^{2} = n\vartheta^{2s} 
\Intrn \Big(\Big( a(\vartheta x) + 
\Frac{\nabla a(\vartheta x)\cdot (\vartheta x)}{n}\Big) F(u) - \lambda\Frac{u^2}{2}\Big).
$$
By condition $(A4)$, we get
\begin{align*}
\Frac{n-2s}{2n}\Intrn |(-\Delta)^{s/2}u|^{2} &< \vartheta^{2s} \Intrn \Big(a_{\infty}F(u)- \lambda\Frac{u^2}{2}\Big)
=\vartheta^{2s} \Intrn G_{\infty}(u).
\end{align*}
Since $u\in \mathcal{P}_{\infty}$, the inequality above yields  $\theta>1$.
\end{proof}

\begin{lem}\label{mini2}
If $u \in \mathcal{P}$, then $\int G_{\infty}(u) >0$.
\end{lem}
\begin{proof}
Let $u\in \mathcal{P}$. Then,  by condition $(A4)$, $u$ satisfies
\begin{equation*}
\Frac{n-2s}{2}\Intrn |(-\Delta)^{s/2}u|^{2} =  n \Intrn \Big(\Big(a(x) + \Frac{\nabla a(x)\cdot x}{n}\Big)F(u) - \lambda\Frac{u^2}{2}\Big) 
<  n \Intrn G_{\infty}(u).
\end{equation*}
Since $\int |(-\Delta)^{s/2}u|^{2}>0$ otherwise $u$ would be constant and hence the zero function as $u\in L^2(\R^n)$, the assertion follows.
\end{proof}

\begin{lem}
\label{other-proj}
If $u \in \mathcal{P}$, then there exists a unique $\vartheta>0$ such that $u(\cdot/\vartheta)\in \mathcal{P}_{\infty}$ and $\vartheta <1$.
\end{lem}

\begin{proof}
Let $u\in \mathcal{P}$, then $\int G_{\infty}(u)>0$ by Lemma~\ref{mini2}. 
By Lemma \ref{pohoprojeta}, there exists a unique $\vartheta >0$ such that $u(\cdot/\vartheta)\in \mathcal{P}_{\infty}$.
We are left with the proof that $\vartheta <1$. Notice that
 $$
 \Frac{n-2s}{2n}\Intrn |(-\Delta)^{s/2}u|^{2} < \Intrn G_{\infty}(u).
 $$
Since $u(\cdot/\vartheta)\in \mathcal{P}_{\infty}$, 
then the assertion follows since $\vartheta>0$ satisfies
 $$
 \vartheta^{2s} = \frac{n-2s}{2n}\frac{\Intrn |(-\Delta)^{s/2}u|^{2}}{\Intrn G_{\infty}(u)} <1.
 $$
 This concludes the proof.
\end{proof}

\noindent
Notice that, as a consequence of the previous results, a given function
$u \in H^s(\Rn)\backslash \{0\}$ can be projected onto the manifolds $\mathcal{P}$ and 
$\mathcal{P}_{\infty}$ if and only if $\int G_{\infty}(u)>0$. We will also need the
following

\begin{lem} \label{thetay}
If $ u\in \mathcal{P}_{\infty}$, then $u(\cdot-y)\in \mathcal{P}_{\infty},$ for all $ y \in \Rn$. Moreover, there exists $\vartheta_y >1$ with
$$
u\Big(\Frac{\cdot-y}{\vartheta_y}\Big)\in \mathcal{P},  \,\,\qquad
\Lim_{|y|\to \infty}\vartheta_y = 1.
$$
\end{lem}

\begin{proof}
If $ u\in \mathcal{P}_{\infty}$, then from translation invariance, we have 
$u(\cdot-y)\in \mathcal{P}_{\infty}$, for all $y \in \Rn$. Furthermore, from Lemma \ref{pohoprojeta2},
there exists $\vartheta_y >1$ such that $u((\cdot-y)/\vartheta_y)\in \mathcal{P}$.
Suppose by contradiction
that there exists a sequence $(y_j)\subset \Rn$ 
with $|y_j|\to +\infty$ and $\vartheta_{y_j}$ converges either to $A > 1$ or $+\infty$.
Let us define
$$
K(\vartheta_{y_j} x + y_j) := a(\vartheta_{y_j}x+y_j) + \Frac{\nabla a(\vartheta_{y_j}x+y_j)\cdot(\vartheta_{y_j}x+y_j)}{n}.
$$ 
From conditions (f1)-(f2) we have 
$0\leq K(\vartheta_{y_j}x + y_j)F(u(x)) < a_{\infty}F(u(x)) \leq  Cu^2(x)$
for a.e.\ $x\in\R^n$ and for some positive constant $C$. Hence, by Lebesgue Dominated Convergence Theorem, we get
\begin{equation}
\label{succ-conv}
\Lim_{j\to \infty}\Intrn \Big(K(\vartheta_{y_j} x + y_j)F(u) - \lambda\Frac{u^2}{2}\Big) =\Intrn G_{\infty}(u).
\end{equation}
But for each $y_j$ it follows that $u(\frac{\cdot-y_j}{\vartheta_{y_j}})\in \mathcal{P}$ with $\vartheta_{y_j}>1$, which means we have
\begin{equation}
\Frac{n-2s}{2}\Intrn |(-\Delta)^{s/2}u|^{2} = n\vartheta^{2s}_{y_j}\Intrn 
\Big(K(\vartheta_{y_j}x + y_j)F(u) - \lambda\Frac{u^2}{2}\Big). 
\label{theta1}
\end{equation}
The right-hand side of formula \eqref{theta1} goes to $+\infty$ or to $nA^{2s}\int G_{\infty}(u)$, 
while the left-hand side is constant. In the first case 
we immediately get a contradiction. In the
second case, as $u\in \mathcal{P}_{\infty}$ and $A>1$, we get a contradiction too.
\end{proof}

\noindent
Under the assumption of Lemma~\ref{thetay}, we have the following

\begin{lem}\label{limitheta}
$\displaystyle\sup_{y\in \Rn} \vartheta_y = \bar{\vartheta}<+\infty$ 
and $\bar{\vartheta}>1$.
\end{lem}

\begin{proof}
From Lemma~\ref{thetay} there is $R>0$ such 
that $|\vartheta_y | \leq 2$ if $|y|>R$. There exists $M>0$ such that 
$\sup\{\vartheta_y:|y|\leq R\}\leq M$. In fact, suppose that there exists a sequence $(y_j)$ with $|y_j|\leq R$ such that 
$\vartheta_{y_j}\to+\infty$ as $j\to\infty$.   As in the previous lemma, 
\eqref{succ-conv} holds.
Therefore, from \eqref{theta1}, it follows
$$
\Frac{n-2s}{2}\Intrn |(-\Delta)^{s/2}u|^{2}  
= n\vartheta^{2s}_{y_j}\Big(\Intrn G_{\infty}(u)+ o_{y_j}(1)\Big).
$$
Since $\vartheta_{y_j}\to+\infty$ and the left-hand side is constant
we get a contradiction and the proof is complete. 
\end{proof}

\begin{lem}\label{infimoP}
 There exists a real number $\hat{\sigma}>0$ such that $\displaystyle\inf_{u\in \mathcal{P}}\int |(-\Delta)^{s/2}u|^{2}\geq \hat{\sigma}$.
\end{lem}

\begin{proof}
Let $u \in \mathcal{P}$,  then $u$ satisfies \eqref{pohozaev}
and by condition $(A4)$, we have
$$
0<\Frac{n-2s}{2}\Intrn |(-\Delta)^{s/2}u|^{2}  < n\Intrn \Big(a_{\infty}F(u) - \lambda\Frac{u^2}{2}\Big).
$$
On the other hand, from condition \eqref{condiF2} with $p = 2n/(n-2s)$,
given $0<\varepsilon < \frac{\lambda}{a_\infty}$, we get
$$
0<\Frac{n-2s}{2n}\Intrn |(-\Delta)^{s/2}u|^{2}  < a_{\infty}C\Vert u\Vert^{2n/(n-2s)}_{2n/(n-2s)}.
$$
for some $C>0$. Using the fractional Sobolev inequality (cf.\ \cite[Theorem 6.5]{DiNezza}), we find $\hat C>0$ with
$$
0 < \Frac{n-2s}{2na_{\infty}C \hat C} < \Big(\Intrn |(-\Delta)^{s/2}u|^{2} \Big)^{2s/(n-2s)},
$$
which yields the assertion with $\hat{\sigma} := ( (n-2s)/(2na_{\infty}C\hat C))^{(n-2s)/2s}>0$.
\end{proof}

\begin{lem}
$p =: \displaystyle\inf_{u\in\mathcal{P}}I(u)>0$.
\end{lem}
\begin{proof}
Let $u\in \mathcal{P}$, then $I(u)$ satisfies
\begin{equation}
\label{riscritt}
I(u) = \Frac{s}{n}\Intrn |(-\Delta)^{s/2}u|^{2} + \Intrn \Frac{\nabla a(x)\cdot x}{n}F(u)
\geq  \Frac{s}{n}\Intrn |(-\Delta)^{s/2}u|^{2} \geq \Frac{s\hat{\sigma}}{n}>0,
\end{equation}
by Lemma \ref{infimoP} and condition $(A3)$. This concludes the proof.
\end{proof}

\vskip2pt
\noindent
If $u\in H^s(\Rn)$ with $\int G_{\infty}(u)>0$ and $\vartheta>0$ is such that $u(\cdot/\vartheta)\in \mathcal{P}_{\infty}$, then 
\begin{equation}
I_{\infty}(u(x/\vartheta)) = \frac{s}{n}\vartheta^{n-2s}\Intrn |(-\Delta)^{s/2}u|^{2}. 
\label{funemP}
\end{equation}

\noindent
Let $c_\infty$ be defined as in \eqref{cinftydef}. Then, we have the following

\begin{lem}\label{pigual}
$p = c_\infty$.
\end{lem}

\begin{proof}
%-------------------------------------------------------------p < or = c infty
Let $w\in H^s(\Rn)$ be a ground state solution to \eqref{auto}. Then
$w\in \mathcal{P}_{\infty}$ and
$I_{\infty}(w) = c_\infty$, by virtue of Theorem \ref{cim}. 
Set $w_y:= w(x-y),$ for any $y\in \Rn$. 
Of course $w_y\in \mathcal{P}_{\infty}$ and $I_{\infty}(w_y)= c_\infty,$
by translation invariance. 
From Lemma \ref{pohoprojeta2} we find a unique $\vartheta_y>1 $ with
$\tilde{w}_y = w_y(\cdot/\vartheta_y)\in \mathcal{P}$. 
 Therefore, we have
\begin{align*}
& |I(\tilde{w}_y) - c_\infty|  = |I(\tilde{w}_y) - I_{\infty}(w_y)|\\
& = \left|\Frac{1}{2}\Intrn |(-\Delta)^{s/2}\tilde{w}_y|^{2} - \Intrn G(x, \tilde{w}_y) - \Frac{1}{2}\Intrn |(-\Delta)^{s/2} w_y|^{2} + \Intrn G_{\infty}(w_y) \right|\\
&= \left| \Frac{1}{2}(\vartheta^{n-2s}_{y}-1)\Intrn |(-\Delta)^{s/2}w_y|^{2} - \Intrn \Big(a(x)F(\tilde{w_y})-\Frac{\lambda\tilde{w}^2_y}{2}\Big)
 + \Intrn \Big(a_{\infty}F(w_y) - \Frac{\lambda w^2_y}{2}\Big) \right|\\
&=  \left| \Frac{1}{2}(\vartheta^{n-2s}_{y}-1)\Intrn |(-\Delta)^{s/2} w|^{2} - \vartheta^n_{y}\Intrn \Big(a(x\vartheta_y+y)F(w)-\Frac{\lambda w^2}{2}\Big) 
 + \Intrn \Big(a_{\infty}F(w) - \Frac{\lambda w^2}{2}\Big) \right|\\
&= \left| \Frac{1}{2}(\vartheta^{n-2s}_{y}-1)\Intrn |(-\Delta)^{s/2} w|^{2} + (\vartheta^n_{y}-1)\Intrn \Frac{\lambda w^2}{2}
 - \vartheta^n_{y}\Intrn a(x\vartheta_y +y)F(w) + \Intrn a_{\infty}F(w)  \right|\\
&\leq \Frac{\vert\vartheta^{n-2s}_{y}-1 \vert}{2}\Intrn |(-\Delta)^{s/2} w|^{2} + \vert \vartheta^n_y-1 \vert \Intrn \Frac{\lambda w^2}{2} +
\Intrn |F(w)||a_{\infty} - \vartheta^{n}_{y}a(x\vartheta_y + y)|. 
\end{align*}
Since $\vartheta_y \to 1$ if $|y|\to+\infty$, we obtain
$$
|F(w)||a_{\infty} - \vartheta^{n}_{y}a(x\vartheta_y + y)|\to 0\quad
\text{as $|y|\to\infty$, a.e.\ in $\Rn$},
\qquad|F(w)||a_{\infty} - \vartheta^{n}_{y}a(x\vartheta_y + y)|\leq C|w|^2,  
$$
for some positive constant $C$ independent of $y$. 
By Lebesgue Dominated Convergence Theorem,
$$
\Intrn |F(w)||a_{\infty} - \vartheta^{n}_{y}a(x\vartheta_y + y)| =o_y(1), \,\quad
\text{as $|y|\to\infty$.}
$$
In turn, we conclude that $|I(\tilde{w}_y) - c_\infty|\leq o_y(1) $, as $|y|\to\infty$. 
Then, $p=\inf_{u\in\mathcal{P}}I(u)\leq c_\infty$.
On the other hand, consider $u\in \mathcal{P}$ and let
$0<\vartheta <1$ by Lemma \ref{other-proj}
be such that $u(\cdot/\vartheta)\in \mathcal{P}_{\infty}$. 
Since $u\in \mathcal{P}$, then
\begin{align*}
I(u) &= \Frac{s}{n}\Intrn |(-\Delta)^{s/2}u|^{2} + \Frac{1}{n}\Intrn \nabla a(x)\cdot xF(u) 
> \Frac{s}{n}\Intrn |(-\Delta)^{s/2}u|^{2} \\
& \geq \frac{s}{n}\vartheta^{n-2s}\Intrn |(-\Delta)^{s/2}u|^{2} 
=I_{\infty}(u(x/\vartheta))\geq \inf_{u\in {\mathcal P}_\infty} I_\infty(u)=m=c_\infty,
\end{align*}
in light of \eqref{funemP}, $(A3)$ and Lemma~\ref{pre-lem2}. 
Hence $p=\inf_{u\in \mathcal{P}}I(u)\geq c_{\infty}$,
which concludes the proof.
\end{proof}

\begin{lem}
\label{naturale}
$\mathcal{P}$ is a natural constraint for \eqref{problema}.
\end{lem}

\begin{proof} 
If $u\in {\mathcal P}$ is a critical point of $I|_\mathcal{P}$, 
there exists $\mu\in \mathbb{R}$ with
$I'(u) + \mu J'(u) =0.$ The proof is complete as soon as we show that
$\mu= 0$. Computing $I'(u)(\varphi) + \mu J'(u)(\varphi)$ for any 
$\varphi\in H^s(\Rn)$ yields
\begin{align*}
0 &=\Intrn (-\Delta)^{s/2} u(-\Delta)^{s/2} \varphi + \lambda u\varphi - \Intrn a(x)f(u)\varphi \\
&+ \mu\Big[(n-2s)\Intrn (-\Delta)^{s/2} u (-\Delta)^{s/2}\varphi - n\Intrn\Big(\Big(a(x)+\Frac{\nabla a(x)\cdot x}{n}\Big)f(u)\varphi-\lambda u\varphi\Big)\Big]. 
\end{align*}
so that $u$ satisfies the equation
\begin{equation*}
(1+\mu(n-2s))(-\Delta)^{s} u + \lambda(1+\mu n)u = \left[(1+\mu n)a(x) + \mu\nabla a(x)\cdot x \right]f(u). 
\end{equation*}
The solutions of this equation satisfy a 
Poho\v zaev identity $Q(u)=0$, where
\begin{equation*}
Q(u) = \Frac{(1+\mu(n-2s))(n-2s)}{2}\Intrn|(-\Delta)^{s/2} u|^2 - n\Intrn \widehat G(x,u) 
-\Intrn x\cdot \widehat G_{x}(x,u),
\end{equation*}
where we have
\begin{align*}
\widehat G(x,u) &= \left((1+\mu n)a(x) + \mu\nabla a(x)\cdot x\right)F(u) - \lambda\Frac{(1+\mu n)}{2}u^2, \\
x \cdot \widehat G_{x}(x,u) &=  \left((1+\mu+\mu n)\nabla a(x)\cdot x + \mu x\cdot {\mathcal H}_a(x)\cdot x\right)F(u).
\end{align*}
Therefore, $Q$ rewrites as follows
\begin{align*}
Q(u)&= \Frac{(1+\mu(n-2s))(n-2s)}{2}\Intrn |(-\Delta)^{s/2} u)|^2   \\
& -  n\Intrn \left((1+\mu n)a(x) + \mu \nabla a(x)\cdot x\right)F(u)-\lambda\Frac{(1+\mu n)}{2}u^2 \\
&-\Intrn \left((1+\mu+\mu n)\nabla a(x)\cdot x + \mu x\cdot {\mathcal H}_a(x)\cdot x \right)F(u) \\
& =  \Frac{(1+\mu(n-2s))(n-2s)}{2}\Intrn |(-\Delta)^{s/2} u|^2  \\
& -  n(1+\mu n)\Intrn \Big(a(x)+ \Frac{\nabla a(x)\cdot x}{n}\Big)F(u) - \lambda\Frac{u^2}{2} \\
& -  (n+1)\mu\Intrn \Big( \nabla a(x)\cdot x +\Frac{x\cdot {\mathcal H}_a(x)\cdot x}{n+1}\Big)F(u).
\end{align*}
Recalling that $u\in \mathcal{P}$ and substituting \eqref{pohozaev} 
in the equation above, it follows that
\begin{align*}
Q(u) &=\Frac{(1+\mu(n-2s))(n-2s)}{2}\Intrn |(-\Delta)^{s/2} u|^2 - (1+\mu n)\Frac{(n-2s)}{2} \Intrn |(-\Delta)^{s/2} u|^2 \\
&-(n+1)\mu\Intrn \Big( \nabla a(x)\cdot x +\Frac{x\cdot {\mathcal H}_a(x)\cdot x}{n+1}\Big)F(u)\\
& = -\mu s(n-2s)\Intrn |(-\Delta)^{s/2} u|^2 
- (n+1)\mu\Intrn \Big( \nabla a(x)\cdot x +\Frac{x\cdot {\mathcal H}_a(x)\cdot x}{n+1}\Big)F(u).
\end{align*}
On the other hand, since $u$ satisfies $Q(u) = 0$, we end up with
$$
-\mu s(n-2s)\Intrn |(-\Delta)^{s/2} u|^2 = (n+1)\mu\Intrn \Big( \nabla a(x)\cdot x +\Frac{x\cdot {\mathcal H}_a(x)\cdot x}{n+1}\Big)F(u).
$$
From $(A5)$ we have that, if $\mu>0$, the right-hand side of the equation is nonnegative as
$$
\nabla a(x)\cdot x +\Frac{x\cdot {\mathcal H}_a(x)\cdot x}{n+1}\geq\frac{n}{n+1}
\Big(\nabla a(x)\cdot x +\Frac{x\cdot {\mathcal H}_a(x)\cdot x}{n}\Big)\geq 0,
$$
while the left-hand side is negative. If $\mu<0$ one gets 
the same contradiction. Whence $\mu=0$.
\end{proof}

%\begin{rem}
%\label{natur-succ}
%If $(u_j)\subset {\mathcal P}$ is a bounded sequence such that 
%$I'|_{{\mathcal P}}(u_j)\to 0$ as $j\to \infty$, then actually 
%$I'(u_j)\to 0$ as $j\to \infty$. In fact, after a density argument,
%one can argue as in the proof of Lemma~\ref{naturale} on the
%expression $I'(u_j)+\mu_j J'(u_j)=o_j(1)$, for some $\mu_j\subset\R$
%leading to
%$$
%\mu_j\Big[ s(n-2s)\Intrn |(-\Delta)^{s/2} u_j|^2 
%+ (n+1)\Intrn \Big( \nabla a(x)\cdot x +\Frac{x\cdot {\mathcal H}_a(x)\cdot x}{n+1}\Big)F(u_j)\Big]=o_j(1).
%$$
%Taking into account assumption $(A5)$ and Lemma \ref{infimoP},
%there exists $\sigma>0$ such that for all $j\geq 1$,
%$$
%s(n-2s)\Intrn |(-\Delta)^{s/2} u_j|^2 
%+ (n+1)\Intrn \Big( \nabla a(x)\cdot x +\Frac{x\cdot {\mathcal H}_a(x)\cdot x}{n+1}\Big)F(u_j)\geq\sigma,
%$$
%which implies $\mu_j\to 0$ and hence the conclusion.
%\end{rem}

\vskip2pt
\noindent
{\em Proof of Theorem~\ref{nonexist} concluded.}
Assume by contradiction that there exists 
a critical point  $z\in H^s(\Rn)$ of 
 $I$ at level $p$. In particular, $z\in \mathcal{P}$ and $I(z)=p$. 
 Let $\vartheta \in (0,1)$ be such that $z(\cdot /\vartheta)\in \mathcal{P}_{\infty}$. Then
\begin{align*}
p = I(z)&= \Frac{s}{n}\Intrn |(-\Delta)^{s/2}z|^{2} + \Frac{1}{n}\Intrn \nabla a(x)\cdot xF(z) \\
&> \Frac{s}{n}\Intrn |(-\Delta)^{s/2}z|^{2}>\frac{s}{n}\vartheta^{n-2s}\Intrn |(-\Delta)^{s/2}z|^{2} \\
&= I_{\infty}(z(\cdot/\vartheta))\geq 
\inf_{u\in {\mathcal P}_\infty} I_\infty(u)=m=c_\infty\;,
\end{align*}
using $(A3)$ and \eqref{funemP}, Lemma~\ref{pre-lem2}
and Theorem~\ref{cim}. Then $p>c_\infty$, contradicting Lemma~\ref{pigual}.
In particular the infimum $p$ is not achieved, otherwise, if $I(v)=p$
and $I'|_{{\mathcal P}}(v)=0$ for some $v\in H^s(\Rn)$, in light
of Lemma~\ref{naturale}, we would have $I'(v)=0$, contradicting the
first part of Theorem~\ref{nonexist}.
\qed

\medskip

\section{Existence results}

In this section we show the existence of a  solution of problem \eqref{problema}.  To this aim, we shall assume that the hypotheses of Theorem~\ref{exist} are satisfied.
As we have seen in the previous sections, we should look for solutions which have energy levels above
$c_{\infty}$. In order to find such a solution we follow some
ideas of \cite{ACR} based upon linking and the barycenter 
function on the Nehari manifold. 
In our case, since the nonlinear terms 
of the equation are not homogeneous, we are led to the 
Poho\v zaev manifold $\mathcal{P}$  and obtain the desired solution by a linking argument. 
We also make use of a barycenter function, similar to that of \cite{ACR} and  used by G.S.\ Spradlin \cite{spradlin,spradlin2} as well. 

\begin{lem}
\label{mpass-cond}
$I$ satisfies the geometrical properties of the Mountain Pass theorem.
\end{lem}
\begin{proof}
On one hand, for the local minimum 
condition at the origin,
by \eqref{condiF2} one can argue exactly  as in the proof of Lemma~\ref{pre-lem1}.
On the other hand, if $w \in H^{s}(\Rn)$ is a least energy solution to \eqref{auto},
by Lemma \ref{first-lem} there exists $\gamma \in \Gamma_{\infty}$ such that $\gamma(t)=w(x/tL)$
for $t>0$ and $L>0$ large enough. In turn,
if $\gamma_y(t):=w((\cdot -y)/tL)$, by $(A2)$ and Lebesgue Dominated Convergence Theorem,
$$
I(\gamma_y(1))=I_\infty(\gamma_y(1))+\int (a_\infty-a(x+y))F(\gamma(1))=
I_\infty(\gamma(1))+o_y(1)<0,
\quad\text{for $|y|$ large},
$$
since $I_\infty(\gamma(1))<0$, concluding the proof.
\end{proof}

\noindent
Let $c$ be the min-max mountain pass level for $I$ 
\begin{equation} \label{minmax}
c= \min_{\gamma\in \Gamma}\max_{t\in [0,1]}I(\gamma(t)),
\qquad 
\Gamma:=\left\lbrace\gamma\in C([0,1],H^s(\Rn)):\, \gamma(0)=0,\, I(\gamma(1))<0\right\rbrace.
\end{equation}
We start by proving 
that the min-max levels of the Mountain Pass Theorem for 
$I$ and $I_\infty$ agree.

\begin{lem} \label{equalc}
$c_\infty = c$.
\end{lem}

\begin{proof}
If $\gamma\in \Gamma$, then $I(\gamma(1))<0$ and 
since $I_{\infty}\leq I$, we have $I_{\infty}(\gamma(1))<0$. Then, $\Gamma\subset \Gamma_\infty$ yielding
$$
c_\infty =\inf_{\gamma\in \Gamma_\infty}\max_{t\in [0,1]}I_{\infty}(\gamma(t))
\leq \displaystyle\inf_{\gamma\in \Gamma}\max_{t\in [0,1]}I_{\infty}(\gamma(t))\leq 
\inf_{\gamma\in \Gamma}\max_{t\in [0,1]}I(\gamma(t))=c.
$$
Let now $\varepsilon>0$ be arbitrary and let 
$\gamma\in \Gamma_\infty$ such that $I_{\infty}(\gamma(t))\leq c_\infty + \varepsilon$, for all $t\in [0,1]$. 
Choose $y\in \Rn$ and  translating 
$\tau_y(\gamma(t))(x):=\gamma(t)(x-y)$ with $|y|$ large enough, we 
get $\tau_y\circ \gamma\in \Gamma$
(see Lemma \ref{mpass-cond}). If $t_y\in [0,1]$ is such that $I(\tau_y(\gamma(t_y)))$ is the maximum value on $[0,1]$ of $t\mapsto I(\tau_y\circ \gamma(t))$, then 
$$
c_\infty+\varepsilon \geq I_{\infty}(\gamma(t_y))=
I_{\infty}(\tau_y\circ\gamma(t_y))=
\max_{[0,1]}I(\tau_y\circ\gamma)\geq 
\inf_{\gamma\in \Gamma}\max_{t\in [0,1]}I(\gamma(t))=c.
$$
This gives $c_\infty \geq c$ by the arbitrariness of $\eps$
and the assertion follows.
\end{proof}

\begin{lem} \label{equalcp}
$p=c$.
\end{lem}

\begin{proof}
The assertion follows by combining Lemma~\ref{pigual} and Lemma~\ref{equalc}.
\end{proof}

\noindent
Now we observe the following property of $\mathcal{P}$ with respect to the paths in the Mountain Pass Theorem.

\begin{lem}\label{inter}
For every $ \gamma \in \Gamma$ there exists $s\in (0,1)$ such that $\gamma(s)\in \mathcal{P}$.
\end{lem}

\begin{proof}
By the proof of Lemma \ref{manifold} $(a)$, we learn that
 there exists $\rho >0$ such that $J(u)>0$ if $0 < \Vert u \Vert_{H^s} < \rho$. Furthermore, we have
\begin{align*}
J(u) &= \Frac{n-2s}{2}\Intrn |(-\Delta)^{s/2} u|^2 - n\Intrn G(x,u) - \Intrn\nabla a(x)\cdot x\, F(u) \\
&= n I(u) - s\Intrn |(-\Delta)^{s/2} u|^2 -\Intrn\nabla a(x)\cdot x\,F(u).
\end{align*}
From $(A3)$ it follows that $J(u) < nI(u),$ for every $u\in H^s(\Rn)\setminus\{0\}$.
If $\gamma \in \Gamma$, 
we have $J(\gamma(0))=0$ and $J(\gamma(1)) <nI(\gamma(1)) <0$.
Then there exists $t\in(0,1)$ with
$\Vert \gamma(t)\Vert_{H^s} > \rho$ and $J(\gamma(t)) = 0$.
\end{proof}

\noindent
We recall that a sequence $(u_j)$ is said to be a Cerami sequence for 
$I$ at level $d $ in $\R$, denoted by $(Ce)_d$, if  $I(u_n)\to d $ and $\Vert I'(u_j)\Vert_{H^{-s}}(1+ \Vert u_j\Vert_{H^s} )\to 0$. We have the following

\begin{lem}\label{climitada}
If $(u_j)$ is a $(Ce)_d$ sequence with $d >0$, then it has a bounded subsequence.
\end{lem}
\begin{proof}
By contradiction, let $\Vert u_j\Vert_{H^s}\to+\infty$. If $\hat{u}_j:=u_j\|u_j\|^{-1}_{H^s} $,
 then $\| \hat{u}_j\|_{H^s} = 1$ and $\hat{u}_j\rightharpoonup \hat{u}$, up to a subsequence. Therefore, one of the two cases occur:
\begin{align*}
&\text{Case} \ 1:\quad
\Limsup_{j\to \infty}\sup_{y\in \Rn}\Int_{B_1(y)}|\hat{u}_j|^2=\delta>0, \\
&\text{Case} \ 2:\quad
\Limsup_{j\to \infty}\sup_{y\in \Rn}\Int_{B_1(y)}|\hat{u}_j|^2=0.
\end{align*}
Suppose Case 2 hold. Fixing $L>2\sqrt{dD}$, with $D$ as in 
assumption (f3), gives
$$
I(L u_j\|u_j\|_{H^s}^{-1}) = \Frac{L^2}{2}- \Intrn a(x)F(L u_j\|u_j\|_{H^s}^{-1}).
$$
 Given $\varepsilon>0$, by inequality \eqref{condiF2} 
 there exists $C_\eps>0$ (here $2<p<2n/(n-2s)$) with
 $$
 \Intrn a(x)F(L u_j\|u_j\|_{H^s}^{-1})< \frac{a_{\infty}\eps L^2}{2}\frac{\|u_j\|^2_2}{\lambda\|u_j\|_2^2+\|(-\Delta)^{s/2} u_j\|_2^2}
 + C_\eps L^p\|\hat{u}_j\|_p^p\leq \frac{a_{\infty}\eps L^2}{2\lambda}+o_j(1),
 $$
 where $\|\hat{u}_j\|_p\to 0$ by a variant of Lions' Lemma \cite[Lemma I.1]{lions}.
 For $\varepsilon = \lambda/(2a_{\infty})$, we have
 $$
 I(L u_j\|u_j\|_{H^s}^{-1}) \geq  \Frac{L^2}{4} - o_j(1).
 $$
We have $L\|u_j\|^{-1}_{H^s}\in (0,1)$ for $j$ large and if we consider
$t_j\in (0,1)$ with $I(t_ju_j) = \max\limits_{t\in[0,1]}I(tu_j)$,
 \begin{equation}
 \label{basso}
 I(t_ju_j)=\max_{t\in [0,1]}I(tu_j)\geq  I(L u_j\|u_j\|_{H^s}^{-1}) \geq  \Frac{L^2}{4} - o_j(1).
 \end{equation}
 On the other hand, using $(f3)$ we obtain
 \begin{align} \label{NQNOVA2}
I(t_ju_j)&= I(t_ju_j) - \Frac{1}{2}I'(t_ju_j)(t_ju_j)=
 \Intrn a(x)\Big( \Frac{1}{2}f(t_ju_j)(t_ju_j) - F(t_ju_j)\Big) \\
 &\leq 
 \nonumber
  D \Intrn a(x)\Big(\Frac{1}{2}f(u_j)u_j - F(u_j)\Big)  =
 \nonumber
  D(I(u_j) - \Frac{1}{2}I'(u_j)u_j))
= D d + o_j(1). 
 \end{align}
 Then, on account of the choice of $L$, combining
 \eqref{basso} and \eqref{NQNOVA2}, we get a contradiction.
In  Case 1, let $(y_j)$ be a sequence such that $|y_j|\to+\infty$ 
and 
\begin{equation}
\label{bassocaso1}
\Int_{B_1(y_j)}|\hat{u}_j|^2 > \delta/2.
\end{equation}
Recalling that 
$\hat{u}_j(\cdot+ y_j)\rightharpoonup \bar{u}$ in $H^s(\Rn)$
as $j\to\infty$,  we obtain
$\int_{B_1(0)} |\bar{u}(x)|^2 > \delta/2,$
namely $\bar{u}\neq 0$. Thus, there exists $\Omega \subset B_1(0)$, with $|\Omega| >0$ such that
\begin{equation}
\label{inft-pf}
0 \neq \bar{u}(x) = \displaystyle \lim_{j\to \infty} \hat{u}_j(x+ y_j) = \displaystyle\lim_{j\to \infty} \Frac{u_j(x+ y_j)}{\Vert u_j\Vert_{H^s}},
\quad\,\,  \text{a.e.\ $x\in \Omega$},
\end{equation}
yielding $u_j(x+ y_j)\to\infty$ for a.e.\ $x\in \Omega$. 
% % % %  prova % % % % % % %
We claim that, actually
$u_j(x+ y_j)\to+\infty$ for $x\in \Omega$. Setting 
$\zeta_j(x):=\hat u_j(x+y_j)$, for a $\mu_j\to 0$ in $H^{-s}(\Rn)$
as $j\to\infty$, we have
$$
(-\Delta)^{s/2}\zeta _j+\lambda \zeta_j=
\frac{a(x+y_j)}{\|u_j\|_{H^s}}f(\|u_j\|_{H^s}\zeta_j)+\frac{\mu_j}{\|u_j\|_{H^s}}.
$$
Testing this equation by $\zeta_j^-$ and taking unto account that
$$
\int \frac{a(x+y_j)}{\|u_j\|_{H^s}}f(\|u_j\|_{H^s}\zeta_j)\zeta_j^-=0,
\qquad \frac{\langle \mu_j,\zeta_j^-\rangle}{\|u_j\|_{H^s}}
=\frac{\langle \mu_j, u_j^-(\cdot+y_j)\rangle}{\|u_j\|_{H^s}^2}
=o_j(1),
$$
by arguing as around formula \eqref{conput-pos}, we conclude that 
$\|\zeta_j^-\|_{H^s}=o_j(1)$ as $j\to\infty$, hence in particular by 
the fractional Sobolev embedding $\|\zeta_j^-\|_{L^p}=o_j(1)$ as $j\to\infty$ for any $2\leq p\leq 2n/(n-2s)$. Since 
$\zeta_j=\hat u_j(\cdot+ y_j)\to \bar{u}$ in $L^p(\Omega)$, we also have
$\zeta_j^-=\hat u_j^-(\cdot+ y_j)\to \bar{u}^-$ in $L^p(\Omega)$. But then
$\bar u^-=0$ on $\Omega$ which means $\bar u>0$ on $\Omega$. In turn,
from~\ref{inft-pf}, we have the claim.
%\footnote{By arguing as in the proof that the weak solutions to the problem
%are nonnegative (namely test $I'(u_j)$ by $u_j^-$, see the computations in
%\eqref{conput-pos}), we obtain that $\|(-\Delta)^{s/2}u_j^-\|_{2}=o_j(1)$ as %$j\to\infty$, which, by the
%fractional Sobolev embedding  (cf.\ \cite[Theorem 6.5]{DiNezza}), entails
%$\|u_j^-\|_{2N/(N-2s)}=o_j(1)$ as $j\to\infty$ and then
%$u_j^-(x)\to 0$ a.e.\ in $\R^N$. Therefore, eventually,
%we have $u_j(x+ y_j)\to+\infty$ for a.e. $x\in \Omega$,
%for if $u_j(x+ y_j)\to-\infty$ on a subset of $\Omega$ of positive measure...}
Thus, by $(f3)$,  Fatou Lemma and $(A1)$, with $\sigma:=\inf_{\Rn} a$, 
\begin{align*}
&\Liminf_{j\to \infty}\Intrn a(x)\Big( \frac{1}{2}f(u_j)u_j - F(u_j)\Big) \\
&=\Liminf_{j\to \infty}\Intrn a(x+y_j)\Big( \frac{1}{2}f(u_j(x+ y_j))u_j(x+y_j) - F(u_j(x+y_j)\Big) \\
&\geq \Liminf_{j\to \infty}\Int_{\Omega} \sigma\Big(\frac{1}{2} f(u_j(x+ y_j))u_j(x+y_j) - F(u_j(x+y_j)\Big) \\
&\geq  \Int_{\Omega}\Liminf_{j\to \infty} \sigma\Big( \frac{1}{2}f(u_j(x+ y_j))u_j(x+y_j) - F(u_j(x+y_j)\Big)=+ \infty.
\end{align*}
On the other hand, 
$|I'(u_j)u_j|\leq \Vert I'(u_j)\Vert_{H^{-s}} \Vert u_j \Vert_{H^s} \to 0,$ 
as $j\to\infty$. Then,
$$
\Intrn a(x)\Big(\Frac{1}{2}f(u_j)u_j - F(u_j)\Big) = 
I(u_j) - \Frac{1}{2}I'(u_j)u_j = d +o_j(1),
$$
which gives a contradiction. If, 
instead, $(y_j)$ in \eqref{bassocaso1} is bounded,  
say $|y_j|\leq R$ for some $R$, we obtain
$$
\Frac{\delta}{2}\leq \Int_{B_1(y_j)}|\hat{u}_j|^2 
\leq \Int_{B_{2R}(0)}|\hat{u}_j|^2,
$$
and since $\hat{u}_j \to \hat  u$ in $L^2(B_{2R}(0))$, it follows that
$$
\delta/2 \leq \Int_{B_{2R}(0)}|\hat u|^2.
$$
Similarly to the previous case, there exists $\Omega \subset B_{2R}(0)$  of positive measure such that \eqref{inft-pf} holds.
The argument follows as above for the case where $(y_j)$ 
is unbounded and we get a contradiction. 
\end{proof}

%-------------------------------------------------------------------

\noindent
The next step is to show the existence of a Cerami sequence for the functional $I$ at level $c$.

\begin{lem}\label{ceramip4}
Let $c$ be as in (\ref{minmax}), then
there exists a $(Ce)_c$ sequence
 $(u_n)\subset H^s(\Rn)$.
\end{lem}

\begin{proof}
We apply the Ghoussoub-Preiss theorem \cite[Theorem 6]{ekeland} with $X = H^s(\Rn)$,  see also  \cite{Preiss}. 
Consider $z_0 = 0$ and  $z_1$ in $H^s(\Rn)$
with $I(z_1)< 0$ (cf.\ Lemma \ref{mpass-cond}). 
Then the  Poho\v zaev manifold $\mathcal{P} $ separates $z_0$ and $z_1$. Indeed, observe that
$z_0 = 0 \notin \mathcal{P}$ and $z_1 \notin \mathcal{P}$, since $J(z_1) < nI(z_1)<0$ (cf.\ proof of $(a)$ of Lemma \ref{inter}). 
Moreover, there exists $\rho>0$ such that, if $0< \Vert u\Vert_{H^s} < \rho$,  then $J(u)>0$ (cf.\ proof of Lemma \ref{manifold}). 
We have $H^s(\Rn)\setminus \mathcal{P}=\left\lbrace 0 \right\rbrace \cup \left\lbrace  J>0\right\rbrace 
\cup\left\lbrace J<0\right\rbrace $.
The ball $B_{\rho}(z_0)$ is in a connected component $C_1$ of 
$\left\lbrace 0 \right\rbrace \cup \left\lbrace 
J>0\right\rbrace$. On the other 
hand, $z_1$ is in a connected component of $\left\lbrace 
J<0\right\rbrace$.
In this setting, we get a sequence $(u_j)\subset H^s(\Rn)$ such that
\begin{equation*}
\delta(u_j,\mathcal{P})\to 0,  \quad
I(u_j)\to  c,  \quad
\Vert I'(u_j)\Vert (1+ \Vert u_j\Vert_{H^s}) \to 0, 
\end{equation*} 
where $\delta$ denotes the geodesic metric on $H^s(\Rn)$, defined by
$$
\delta(u,v):= \text{inf} \Big\{ 
\int_0^1\displaystyle\frac{\|\gamma'(\sigma)\|_{H^s}}{1 + \|\gamma(\sigma)\|_{H^s}}d\sigma :\, \gamma \in C^1([0,1], H^s(\Rn) ), \,\,
\gamma(0) = u, \,\,\gamma(1)=v \Big\}.
$$
This complets the proof.
\end{proof}

\noindent
For the following type of properties, we refer the reader to the book \cite{tinta}.

\begin{lem}\label{splitting4}
Let $(u_j)\in H^s(\Rn)$ be a bounded sequence such that 
$$
I(u_j)\rightarrow d >0 \quad \text{and} \quad \Vert I'(u_j)\Vert_{H^{-s}}(1+ \Vert u_j\Vert_{H^s})\rightarrow 0\;.
$$ 
Replacing $(u_j)$ by a subsequence, if necessary, there exists a 
solution $\bar{u}$ of \eqref{problema}, a number  
$k\in \mathbb{N}\cup \left\lbrace 0\right\rbrace$, $k$ functions 
$u^1,u^2,\ldots,u^k$ and $k$ sequences of points $y^1_j,y^2_j,\ldots,y^k_j\in\Rn$, satisfying:
\begin{itemize}
\item[a)] $u_j \rightarrow \bar{u}$ in $H^s(\Rn)$ or
\item[b)] $u^i\in H^s(\Rn)$ are positive solutions to \eqref{auto} radially symmetric about some point;
\item[c)] $|y^{i}_{n}|\to+\infty$ and $|y^{i}_{n} - y^{m}_{n}|\to+\infty, \,\, i\neq m$;
\item[d)] $u_j - \displaystyle\sum_{i=1}^{k}u^i(x-y^{i}_{j})\rightarrow \bar{u}$;
\item[e)] $I(u_j)\rightarrow I(\bar{u}) + \displaystyle\sum_{i=1}^{k}I_{\infty}(u^i)$.
\end{itemize}
\end{lem}
\noindent
That the solutions $u^i\in H^s(\Rn)$  to \eqref{auto} are positive 
and radially symmetric about some point follows from \cite[Theorem 1.3]{felmer}, namely
a Gidas-Ni-Niremberg type result in the fractional case
($u^i\neq 0$, $u^i\geq 0$ and hence $u^i>0$, see \cite{felmer}).

%\begin{rem}
%I didn't do the splitting, but I am assuming that is possible.
%\end{rem}

\begin{cor}\label{44}
If $I(u_j)\rightarrow c_{\infty}$ and $\Vert I'(u_j)\Vert_{H^{-s}}(1+ \Vert u_j\Vert_{H^s})\to 0$, 
then either $(u_j)$ is relatively compact in $H^s(\Rn)$ or Lemma~\ref{splitting4} 
holds with $k=1$ and $\bar{u} = 0$.
\end{cor}

\noindent
Let us set 
$$
c_\sharp:=\inf\big\{c>c_\infty: \text{$c$ is a radial critical value of $I_\infty$}\big\}.
$$
Then we have the following
\begin{lem}\label{split}
Assume that 
\begin{equation}
\label{nondeg}
\text{$c_\infty$ is an isolated radial critical level for $I_\infty$},
\end{equation}
Then $c_\sharp>c_\infty$  and $I$ satisfies condition $(Ce)$ at level $d\in (c_\infty, \min\{c_\sharp,2c_\infty\})$.
Assume now that
\begin{equation}
\label{uniq-ass}
\text{the limiting problem~\eqref{auto} admits a unique positive radial solution}.
\end{equation}
Then $I$ satisfies condition $(Ce)$ at level $d\in (c_\infty,2c_\infty)$.
\end{lem}

\begin{proof}
Take a sequence $(u_j)\in H^s(\Rn)$ such that $I(u_j)\rightarrow d$ and 
$\Vert I'(u_j)\Vert_{H^{-s}}(1+ \Vert u_j\Vert_{H^s})\rightarrow 0$ as $j\to\infty$.  
By Lemma \ref{climitada}, $(u_j)$ has a bounded subsequence. 
Applying Lemma \ref{splitting4}, up to subsequences, we have
$$
u_j - \displaystyle\Sum_{i=1}^{k}u^i(x-y^i_j) \rightarrow \bar{u}\quad \text{in $H^s(\Rn)$},\qquad
I(u_j)\to  I(\bar{u}) + \displaystyle\Sum_{i=1}^{k}I_{\infty}(u^i),
$$
where $u^i$ is a solution to \eqref{auto}, 
$|y^i_j|\to+\infty$ and $\bar{u}$ is a (possibly zero) solution of \eqref{problema}. 
Since $d<2c_{\infty}$, then $k<2$. If $k=1$, we have two cases to distinguish.
\vskip2pt
\noindent
Let us first assume that \eqref{nondeg} holds. Then $c_\sharp>c_\infty$,
otherwise there exists a sequence $c_j$ of radially symmetric (about some point) critical values of $I_\infty$
such that $c_j>c_\infty$ and $c_j\to c_\infty$ as $j\to\infty$. 
%On the other 
%hand, the corresponding critical points $u_j\in H^s(\Rn)$, 
%solutions to \eqref{auto}, are positive and radially symmetric about some point as follows from \cite[Theorem 1.3]{felmer}. Hence we have found 
 %a sequence $c_j$ of radial critical values of $I_\infty$
%with $c_j>c_\infty$ and $c_j\to c_\infty$ as $j\to\infty$,
%contradicting \eqref{nondeg}.
\vskip2pt
\noindent
$\bullet$ $\bar{u}\neq 0$, which implies $I(\bar{u})\geq p=c_\infty$ 
and hence $I(u_j)\geq 2c_\infty$.
\vskip2pt
\noindent
$\bullet$ $\bar{u} = 0$, which yields $I(u_j)\rightarrow I_{\infty}(u_1)$.
If $I_{\infty}(u_1)=c_\infty$, we have a contradiction. 
If $I_{\infty}(u_1)=\tilde c>c_\infty$, then $I_{\infty}(u_1)\geq c_\sharp\geq \min\{c_\sharp,2c_\infty\}$,
against $d<\min\{c_\sharp,2c_\infty\}$.
Then $k=0$ and $u_j \rightarrow \bar{u}$.
\vskip2pt
\noindent
Let us now assume that \eqref{uniq-ass} holds.
\vskip2pt
\noindent
$\bullet$ $\bar{u}\neq 0$, which implies $I(\bar{u})\geq p=c_\infty$ 
and hence $I(u_j)\geq 2c_\infty$.
\vskip2pt
\noindent
$\bullet$ $\bar{u} = 0$, which yields $I(u_j)\rightarrow I_{\infty}(u_1)=c_\infty$.
The fact that $I_{\infty}(u_1)=c_\infty$ follows by using uniqueness 
assumption \eqref{uniq-ass}.
These conclusions go against the
assumption $c_\infty<d<2c_\infty$.
\end{proof}

\begin{lem} \label{Ibounded}
Let $I(u_j) \to d > 0$ and $\{u_j\} \subset \mathcal{P}$. Then $\{u_j\}$ is bounded in $H^s(\Rn)$.
\end{lem}
\begin{proof}
If $u_j \in \mathcal{P},$ then using $(A3)$ and 
the first equality of \eqref{riscritt}, we get
$$
d + 1 \geq I(u_j) \geq \frac{s}{n}\int |(-\Delta)^{s/2} u_j|^2.
$$
In turn, by the fractional Sobolev inequality, the sequence $\|u_j\|_{2n/(n-2s)} $ is also bounded.
By \eqref{condiF2} with $\varepsilon < \lambda/\|a\|_\infty$, we have
$$
\int a(x) F(u_j) \leq \frac{1}{2}\eps\Vert a \Vert_\infty 
\Vert u_j \Vert_2^2 + C_\varepsilon \Vert u_j\Vert_{2n/(n-2s)}^{2n/(n-2s)}.
$$
Replacing this in the expression of $I$
\begin{equation*}
d + 1 \geq I(u_j) \geq \frac{1}{2}\int |(-\Delta)^{s/2} u_j|^2 +  
\frac{1}{2} ( \lambda
-\eps\Vert a \Vert_\infty  )
\Vert u_j \Vert_2^2 - C_\varepsilon \Vert u_j\Vert_{{2n/(n-2s)}}^{{2n/(n-2s)}},
\end{equation*}
so $\|u_j\|_2$ is bounded as well, and the assertion follows.
\end{proof}

\noindent
Next, we introduce the barycenter function.

\begin{defi} 
Define the barycenter function  of a $u\in H^s(\Rn)\setminus\{0\}$ by setting
$$
\mu(u)(x):= \Frac{1}{|B_1|}\displaystyle\int_{B_1(x)}|u(y)|dy.
$$
It follows that $\mu(u)\in L^{\infty}(\Rn)\cap C(\R^n)$. Subsequently, take
 $$
 \hat{u}(x):= \left[\mu(u)(x) - \Frac{1}{2}\max \mu(u)\right]^+.
 $$ 
 It follows that $\hat{u}\in C_0(\Rn)$. Now define the barycenter  of $u$ by
$$
\beta(u) = \Frac{1}{\|\hat{u}\|_{L^1}}\Intrn x\hat{u}(x)dx\in \Rn.
$$
\end{defi}

\noindent
Since $\hat{u}$ has compact support, by definition, $\beta(u)$ is well defined. 
$\beta$ satisfies the following properties:
\begin{itemize}
\item[$(a)$] $\beta$ is a continuous function in $H^s(\Rn)\setminus\left\lbrace 0 \right\rbrace$.
\item [$(b)$] If $u$ is radially symmetric, then $\beta(u) = 0$.
\item[$(c)$] Given $y\in \Rn$ and setting $u_y(x):= u(x-y)$, then $\beta(u_y) = \beta(u)+y$.
\end{itemize}

\vskip3pt
\noindent
We shall also need the following 
\begin{lem}
\label{I2-boundd}
Assume that $u_j,v_j\subset H^s(\Rn)$ are such that $\|u_j -v_j \|_{H^s} \to 0$
and $I'(v_j)\to 0$ as $j\to\infty$. Then, $I'(u_j)\to 0$ as $j\to\infty$
\end{lem}
\begin{proof}
By assumption (1) of Theorem~\ref{exist}, we have $f\in {\rm Lip}(\R,\R^+)$.
Observe first that, for every $w,\varphi,\psi\in H^s(\Rn)$, we have
\begin{equation}
\label{secondder-form}
I''(w)(\varphi,\psi)=\int (-\Delta)^{s/2}\varphi(-\Delta)^{s/2}\psi
+\lambda\int \varphi\psi-\int a(x)f'(w)\varphi\psi.
\end{equation}
Also, by the Mean Value Theorem, for any $u,v\in H^s(\Rn)$ 
and $\varphi\in H^s(\Rn)$, there exists $\xi\in (0,1)$ with
\begin{equation*}
I'(v)(\varphi)-I'(u)(\varphi)=I''(u+\xi(v-u))(\varphi,v-u).
\end{equation*}
Therefore, by taking into account that $|f'(u_j+\xi_j(v_j-u_j))|\leq C$ 
a.e.\ and for every $j\geq 1$
by assumption (f1), for all $j\geq 1$ we find $\xi_j\in (0,1)$ 
such that from formula \eqref{secondder-form} we obtain
\begin{align*}
 I'(v_j)(\varphi)-I'(u_j)(\varphi) &=I''(u_j+\xi_j(v_j-u_j))(\varphi,v_j-u_j) \\
& =\int (-\Delta)^{s/2}\varphi(-\Delta)^{s/2} (v_j-u_j)
+\lambda\int \varphi (v_j-u_j) \\
&-\int a(x)f'(u_j+\xi_j(v_j-u_j))\varphi (v_j-u_j) \\
&\leq C\|\varphi\|_{H^s}\|v_j-u_j\|_{H^s}+Ca_\infty\int|\varphi||v_j-u_j| 
\leq C\|\varphi\|_{H^s}\|v_j-u_j\|_{H^s}.
\end{align*}
In turn, taking the supremum over the $\varphi\in H^s(\Rn)$ with 
$\|\varphi\|_{H^s}\leq 1$, we get as $j\to\infty$
\begin{equation*}
\| I'(v_j)-I'(u_j)\|_{H^{-s}}\leq C\|v_j-u_j\|_{H^s}=o_j(1),
\end{equation*}
which concludes the proof.
\end{proof}

%\noindent
%We will need the following 
%\begin{lem}
%\label{bar-conv}
%Let $u_j,v_j\subset H^s(\Rn)$ with $\| \mu(u_j) - \mu(v_j) \|_{L^\infty} \to 0$.
%Then $||\beta(u_j)|-|\beta(v_j)||\to 0$.
%\end{lem}
%\begin{proof}
%\color{red} Lili, could you please add a short proof of this?}
%\end{proof}

\noindent
Now we define
\begin{equation} \label{defineb}
b:= \displaystyle\inf\left\lbrace I(u): \;u\in \mathcal{P} \; \text{and}\; \beta(u)=0\right\rbrace.
\end{equation}
It is clear that $b\geq c_{\infty}$. Moreover, we have the following

\begin{lem}\label{delta}
$b > c_{\infty}$.
\end{lem}

\begin{proof}
Suppose $b = c_{\infty}$.  By definition, there exists a sequence 
$\{u_j\}$ with $u_j\in \mathcal{P}$ and $\beta(u_j)=0$ such that $I(u_j) \to b.$
By Lemma \ref{Ibounded}, $\{u_j\}$ is bounded.
Since $b=p$ by Lemmas \ref{equalc} and \ref{equalcp}, then $\{u_j\}$ 
is also a minimizing sequence of $I$ on $\mathcal{P}$. 
By Ekeland Variational Principle, there exists another sequence
$\{\tilde u_j\} \subset \mathcal{P}$ such that
$I(\tilde u_j) \to p$, $I'|_\mathcal{P} (\tilde u_j) \to 0$ and $\Vert \tilde u_j - u_j \Vert_{H^s} \to 0$
as $j\to\infty$.
Let us now prove that $I' (\tilde u_j) \to 0$, as $j\to\infty$.
Suppose by contradiction that this is not the case. 
Then, there exists $\sigma>0$ and a subsequence $\{\tilde u_{j_k}\}$ with
$$
\| I' (\tilde u_{j_k})\|>\sigma,
\,\,\quad\text{for all $k\geq 1$ large}.
$$
Arguing as in the proof of Lemma~\ref{I2-boundd}, there exists a 
positive constant $C$ such that
\begin{equation*}
|I'(\tilde u_{j_k})(\varphi) - I'(v)(\varphi)|\leq C\|\tilde u_{j_k} -v\|_{H^s} \|\varphi\|_{H^s},\quad\text{for 
all $k\geq 1$ and any $v,\varphi\in H^s(\Rn)$}.
\end{equation*}
Taking the supremum over $\|\varphi\|_{H^s}\leq 1$ 
yields $\|I'(\tilde u_{j_k}) - I'(v)\|_{H^{-s}}\leq C\|\tilde u_{j_k} -v\|_{H^s}$ for all $k\geq 1$ and any $v\in H^s(\R^n)$. Therefore, if
$\|\tilde u_{j_k} -v\|_{H^s} < \tilde \delta/C:= 2\delta,$ then we have
$\| I'(\tilde u_{j_k}) - I'(v) \|_{H^{-s}} < \tilde \delta.$ for all $v\in H^s(\Rn)$ and $k\geq 1$. This yields,
$\sigma- \tilde \delta < \|I'(\tilde u_{j_k})\|_{H^{-s}}  - \tilde \delta < \|I'(v)\|_{H^{-s}},$ for all $k\geq 1$ large.
For $\tilde \delta\in (0,\sigma)$, we have $\lambda:= \sigma- \tilde \delta >0$ and
$$
\forall v\in H^s(\R^n): \quad  v \in B_{2 \delta}(\tilde u_{j_k})
\,\,\,\Longrightarrow\,\,\,  \|I'(v)\|_{H^{-s}}> \lambda.
$$
Let us now set $\eps:= \min \{p/2, \lambda \delta/8\}$ and $S:=\{\tilde u_{j_k}\}$.
Then, by virtue of \cite[Lemma 2.3]{willem}, there is a deformation $\eta:[0,1]\times H^s(\Rn)\to H^s(\Rn)$ 
at the level $p$, such that
$$
\eta(1,I^{p+\eps}\cap S)\subset I^{p-\eps},\qquad
I(\eta(1,u)) \leq I(u), \,\,\quad \text{for all $u \in H^s(\Rn)$}.
$$
For $k$ large enough, since $\tilde u_{j_k}$ is minimizing for $p$, we have
\begin{equation}
\label{stimamax}
\max_{t >0} I(\tilde u_{j_k}(\cdot/t))=  I(\tilde u_{j_k}) < p +\eps.
\end{equation}
Observe that, for each $k\geq 1$, by $(A4)$ we have
$$
\int G_\infty (\tilde u_{j_k})\geq \int \Big(\Big(a(x)+\frac{\nabla a(x)\cdot x}{n}\Big)F(\tilde u_{j_k})-\lambda\frac{\tilde u_{j_k}^2}{2}\Big)
=\frac{n-2s}{2n}\int |(-\Delta)^{s/2} \tilde u_{j_k}|^2>0,
$$
so that the arguments of Lemma~\ref{pohoprojeta} work for $\tilde u_{j_k}$.
Since $\tilde u_{j_k} \in \mathcal{P}$, the first equality in \eqref{stimamax}
is justified by means of formula \eqref{formPsi} of Lemma~\ref{pohoprojeta} on $\Psi'$, by 
the uniqueness of positive zeros of $\Psi'$ and since $\Psi(\vartheta)>0$ for 
$\vartheta$ small and $\Psi(\vartheta)<0$ for $\vartheta$ large. 
Then, we can infer that
$$
\max_{t >0} I(\eta (1, \tilde u_{j_k}(\cdot/t))< p -\eps.
$$
On the other hand, for $k$ and $L$ fixed large, $\gamma(t):=\eta (1, \tilde u_{j_k}(\cdot/Lt))$ 
is a path in $\Gamma$ since by \eqref{basic}
\begin{align*}
I(\gamma(1)) &=I(\eta (1, \tilde u_{j_k}(\cdot/L)))\leq  I(\tilde u_{j_k}(\cdot/L)) 
= \Frac{L^{n-2s}}{2}\Intrn |(-\Delta)^{s/2}\tilde u_{j_k}|^{2}  - L^n\Intrn \Big( a(L x)F(\tilde u_{j_k}) - \lambda\Frac{\tilde u^2_{j_k}}{2}\Big) \\
&= \Frac{L^{n-2s}}{2}\Intrn |(-\Delta)^{s/2}\tilde u_{j_k}|^{2}  - L^n \Big(\Intrn G_\infty(\tilde u_{j_k})+o_L(1)\Big)<0, \quad\text{for $L\to\infty$}.
\end{align*}
Hence, we deduce that
$$
c \leq \max_{t\in [0,1]} I(\eta (1, \tilde u_{j_k}(\cdot/Lt))
=\max_{t >0} I(\eta (1, \tilde u_{j_k}(\cdot/t))< p -\eps< p,
$$
contradicting that fact that $p=c$, provided by Lemma~\ref{equalcp}.
By Lemma~\ref{I2-boundd}, being $\Vert \tilde u_j - u_j \Vert_{H^s} \to 0$, we get $I' (u_j) \to 0$ as $j\to\infty$.
%Moreover,  $\Vert \tilde u_j - u_j \Vert_{H^s} \to 0$ imply that 
%\begin{align*}
%| \mu(u_j)(x) - \mu(\tilde u_j(x)) |
%&=
%\Big| \Frac{1}{|B_1|}\displaystyle\int_{B_1(x)}|u_j(y)|dy 
%- \Frac{1}{|B_1|}\displaystyle\int_{B_1(x)}|\tilde u_j(y)|dy \Big| \\
%& \leq
%\Frac{1}{|B_1|} \displaystyle\int_{B_1(x)} | |u_j(y)| - |\tilde u_j(y)| | dy \\
%&\leq 
%\Frac{1}{|B_1|} \displaystyle\int_{B_1(x)} | u_j(y) - \tilde u_j(y) | dy \\
%&\leq 
%\| u_j - \tilde u_j \|_{L^2}
%\leq
%\| u_j - \tilde u_j\|_{H^s}
%= o_j(1),
%\end{align*}
%uniformly in $x \in \Rn$.
%Then $\| \mu(u_j) - \mu(\tilde u_j) \|_{L^\infty} \to 0$ and, 
%by Lemma~\ref{bar-conv}, we get $||\beta(\tilde u_j)|-|\beta(u_j)||\to 0$ as $j\to \infty$.
%In turn, since $\beta(u_j)=0$, we have that $ \beta(\tilde u_j) $ is bounded. 
%\footnote{boundeness of $ \beta(\tilde u_j) $ unclear}
Therefore, $\{u_j\}$ satisfies the assumptions of Corollary \ref{44} and since 
$p=c_\infty$ is not attained by Theorem~\ref{nonexist}, 
then the splitting lemma holds with $k=1$,
see Corollary~\ref{44}. This yields
$u_j(x)=u^1(x-y_j)+o_j(1)$ as $j\to\infty$
where $y_j \in \Rn$, $|y_j| \to +\infty$ 
and $u^1$ is a solution of the problem at infinity.
By making a translation, $u_j(x+ y_j) = u^1(x) + o_j(1).$
Applying the barycenter map yields
$\beta(u_j(x+y_j)) = \beta(u_j) - y_j =-y_j$ and
$\beta(u^1(x)+ o_j(1))=\beta(u^1(x))+o_j(1)$
by continuity.  Then, we reach a contradiction, yielding $b > c_{\infty}$.
\end{proof}

\noindent
Let us consider  a positive, radially symmetric, ground state solution $w \in H^s(\Rn)$
to the autonomous problem at infinity. We define the operator $\Pi : \Rn \rightarrow \mathcal{P}$ by
$$
\Pi[y](x) := w \Big(\Frac{x-y}{\vartheta_y}\Big),
$$
where $\vartheta_y$ projects $w(\cdot - y)$ onto $\mathcal{P}$.  
$\Pi$  is continuous as $\vartheta_y$ is unique and $\vartheta_y (w(\cdot - y))$ is a continuous function of $w(\cdot - y)$.

\begin{lem}\label{Pi}
$\beta(\Pi[y](x)) = y$ for every $y\in\Rn$.
\end{lem}

\begin{proof}
Let $v(x) = w ((x-y)/\vartheta_y)$, then
\begin{equation*}
\mu(v)(x)=\Frac{1}{|B_1|}\Int_{B_1(x-y)}\Big| w \Big(\Frac{\xi}{\vartheta_y}\Big)\Big| d\xi
= \mu\Big(w \Big(\Frac{\cdot}{\vartheta_y}\Big)\Big)(x-y),
\end{equation*}
and further, that $\hat{v}(x) = \widehat{w(\cdot/\vartheta_y)}(x-y)$. 
Using the fact that $\|\hat v\|_{L^1}
=\|\widehat{w(\cdot/\vartheta_y)}\|_{L^1}$, we get
\begin{align*}
\beta(v)& =  \Frac{1}{\|\hat{v}\|_{L^1}}\Intrn x \widehat{w(\cdot/\vartheta_y)}(x-y)dx\\
&=\Frac{1}{\|\hat{v}\|_{L^1}}\Intrn (z+y) 
\widehat{w(\cdot/\vartheta_y)}(z)dz\\
&=\Frac{1}{\|\hat{v}\|_{L^1}}\Intrn z \widehat{w(\cdot/\vartheta_y)}(z)dz + \Frac{1}{\|\hat{v}\|_{L^1}}\Intrn y \widehat{w(\cdot/\vartheta_y)}(z)dz\\
&= \beta (w(\cdot/\vartheta_y)) + \Frac{y}{\|\hat{v}\|_{L^1}}\Intrn \hat{v}(y+z)dz= y,
\end{align*}
since $w$ is radially symmetric.
\end{proof}

\begin{lem}\label{limitepi}
$I(\Pi[y])\searrow c_\infty,$ if $|y|\to+\infty$.
\end{lem}

\begin{proof}
Since $\Pi[y]\in \mathcal{P}$, as observed in \eqref{riscritt},
the functional $I$ can be written as
$$
I(\Pi[y]) = \Frac{s}{n} \int |(-\Delta)^{s/2} w \Big(\Frac{x-y}{\vartheta_y}\Big) |^2
+ \Frac{1}{n}\Intrn \nabla a(x)\cdot x F\Big( w \Big(\Frac{x-y}{\vartheta_y}\Big)\Big).
$$
Moreover, since $w \in \mathcal{P}_\infty$, 
by \eqref{funemP} we have $I_{\infty}(w)= \frac{s}{n}\int |(-\Delta)^{s/2} w |^2$ and we obtain
\begin{align*}
I(\Pi[y]) & =  \Frac{s\vartheta^{n-2s}_y}{n}\int |(-\Delta)^{s/2} w |^2\\
&+ \Frac{\vartheta^n_y}{n}\Intrn\nabla a(\vartheta_y x+y)\cdot (\vartheta_y x+y)F(w) \\
&= \vartheta^{n-2s}_yI_{\infty}(w) + \Frac{\vartheta^n_y}{n}\Intrn\nabla a(\vartheta_y x+y)\cdot (\vartheta_y x+y)F(w)
\,\,\,\, (>c_\infty).
\end{align*}
By Lebesgue Dominated Convergence Theorem,  \eqref{antigaA5} and $\vartheta_y\rightarrow 1$ if $|y|\to+\infty$, we get
$$
\Lim_{|y|\to\infty}\Intrn\nabla a(\vartheta_y x+y)\cdot (\vartheta_y x+y)F(w)= 0.
$$
Therefore, $I(\Pi[y])\searrow c_\infty$ if $|y|\to+\infty$ and the 
proof is complete.
\end{proof}

\begin{lem}\label{A7}
Let $C$ be a positive constant such that $|F(s)| \leq C s^2$.  Assume 
\begin{itemize}
\item[$(A6)$]  $\|a_\infty - a\|_{L^\infty} < \Frac{\min\{c_\sharp,2c_\infty\}-c_\infty}{\widehat{\vartheta}^n\Vert w\Vert^2_{2}C},$
\qquad $\widehat{\vartheta}= \displaystyle\sup_{y\in\Rn}\vartheta_y$.
\end{itemize}
Then $I(\Pi[y]) < \min\{c_\sharp,2c_\infty\}$ for every $y\in\Rn$.
\end{lem}

\begin{proof}
The maximum of $t\mapsto I_{\infty}\left(w(\cdot /t)\right)$ is attained at $t=1$. Since
$\vartheta_y >1$, using $(A6)$, we obtain
\begin{align*}
I(\Pi[y]) & =  I_{\infty}(\Pi[y]) + I(\Pi[y]) - I_{\infty}(\Pi[y])\leq I_{\infty}(w) + \Intrn (a_{\infty} - a(x))F(\Pi[y])\\
& <  c_{\infty} + \Frac{\min\{c_\sharp,2c_\infty\}-c_\infty}{\widehat{\vartheta}^n\Vert w \Vert^2_2 C}\Intrn C w^2\left( \Frac{x-y}{\vartheta_y}\right) \\
& =  c_{\infty} + \Frac{(\min\{c_\sharp,2c_\infty\}-c_\infty) \vartheta^n_y}{ \widehat{\vartheta}^n\Vert w\Vert^2_2}\Vert w\Vert^2_2
=\min\{c_\sharp,2c_\infty\},
\end{align*}
which concludes the proof.
\end{proof}

\begin{rem}
Replacing $(A6)$ with 
$\|a_\infty - a\|_{L^\infty} < c_\infty\widehat{\vartheta}^{-n}\Vert w\Vert^{-2}_{2}C^{-1},$
one gets $I(\Pi[y]) <2c_\infty$.
\end{rem}

\noindent
We will need a version of the Linking Theorem with Cerami condition by 
\cite[Theorem 2.3]{bartolobencifortunato}.

\begin{defi}
Let $S$ be a closed subset of a Banach space $X$ and $Q$ a sub manifold of $X$ with relative boundary $\partial Q$. We say that $S$ and $\partial Q$ link if
the following facts hold
\begin{itemize}
\item[1)] $S\cap \partial Q = \emptyset$;
\item[2)] for any $h \in C^0(X,X)$ 
with $h\vert_{\partial Q} = id$, then $h(Q) \cap S \neq \emptyset$.
\end{itemize}
Moreover, if $S$ and $Q$ are as above and $B$ is a subset of  $C^0(X,X)$, then $S$ and $\partial Q$ link with respect to $B$ if $1)$ and 
$2)$ hold for any $h\in B$.
\end{defi}

\begin{thm}
\label{link}
Suppose that $I\in C^1(X,\mathbb{R})$ is a functional satisfying $(Ce)$ condition. 
Consider a closed subset $S\subset X$ and a submanifold $Q\subset X$ with relative boundary $\partial Q$ such that
\begin{itemize}
\item[a)] $S$ and $\partial Q$ link;
\item[b)] $\alpha = \displaystyle \inf_{u\in S}I(u) > \displaystyle\sup_{u\in \partial Q}I(u) = \alpha_0$.
\item[c)] $\displaystyle\sup_{u\in Q}I(u) <+\infty$.
\end{itemize}
If 
$B = \{ h\in C^0(X,X): h\vert_{\partial Q} = {\rm id}\},$
then $\tau= \displaystyle\inf_{h\in B}\sup_{u\in Q}I(h(u))\geq \alpha$ 
is a critical value of $I$.
\end{thm}

%\begin{rem}
%We also refer to the works of E. A. B. Silva \cite{Silva} and M. Schechter \cite{Schechter} for similar versions of the  Linking Theorem with Cerami condition.
%\end{rem}

%%%%%%%%%%%%%%%%%%%%%proof of main theorem

\noindent
{\em Proof of Theorem \ref{exist} concluded.}
%Since $I_{\infty}(u)< I(u)$ for every $u\in H^s(\Rn)\setminus\{0\}$,
%in turn, we have $I_{\infty}(\Pi[y])< I(\Pi[y])$ for every $y\in \Rn$. 
We follow the argument in \cite[Theorem 7.7]{ACR}.
Since we have $b > c_\infty$ from Lemma \ref{delta} and $I(\Pi[y])\searrow c_\infty$ if $|y|\rightarrow \infty$ from Lemma \ref{limitepi}, there exists $\bar{\rho} >0$ such that
\begin{equation}
c_\infty < \max_{|y|=\bar\rho} I(\Pi[y]) < b. \label{final}
\end{equation}
In order to apply the linking theorem, we take
$$
Q := \Pi(\overline{B_{\bar{\rho}}(0)}),\qquad
S:= \left\lbrace u\in H^s(\Rn):  u\in \mathcal{P}, \; \beta(u)=0\right\rbrace,
$$ 
and we show that $\partial Q$ and $S$ link
with respect to $\mathcal H = \left\lbrace h 
\in C(Q,\mathcal{P}): h\vert_{\partial Q} = id\right\rbrace.$
Since $\beta(\Pi[y])=y$ from Lemma \ref{Pi}, we have that $\partial Q\cap S = \emptyset$, as if $u\in S$, then $\beta(u)=0$, and if $u\in \partial Q$, 
$u=\Pi[y]$ for some $y\in\R^n$ 
with $|y|=\bar{\rho}$ and then $\beta(u)=y\neq 0$. 
Now we show that $h(Q) \cap S \neq \emptyset$ for any $h\in \mathcal H$.
Given $h\in \mathcal H$, let 
$T:\overline{B_{\bar{\rho}}(0)} \to\Rn$ 
by defined by $T(y) = \beta\circ h\circ \Pi[y]  $. The function $T$ is continuous, by composition. Moreover, for $|y|=\bar{\rho}$, we have that $\Pi[y]\in \partial Q$, thus $h\circ \Pi[y] = \Pi[y]$, as $h\vert_{\partial Q} = id$, and hence  $T(y)=y$ by Lemma~\ref{Pi}. By Brower Fixed Point Theorem there is 
$\tilde{y}\in B_{\bar{\rho}}(0)$ with $T(\tilde{y})=0$, which implies $h(\Pi[\tilde{y}])\in S$. Then 
$h(Q)\cap S \neq \emptyset$ and $S$ and $\partial Q$ link.
Now, from \eqref{final}, we may write 
$$
b = \displaystyle \inf_{S}I > \displaystyle  \max_{\partial Q} I 
$$
Let us define
$$
d = \displaystyle\inf_{h\in \mathcal H}\max_{u\in Q}I(h(u)).
$$
It is $d\geq b$. 
In fact, if $h\in \mathcal H$, there exists $w\in S$ with 
$w = h(v)$ for some $v\in \Pi(\overline{B_{\bar{\rho}}(0)})$. 
Therefore,
$$
\displaystyle\max_{u\in Q}I(h(u))\geq I(h(v)) = I(w) \geq \displaystyle\inf_{u\in S}I(u)=b,$$ 
and hence $d\geq b$,
which implies $d> c_\infty$. 
Furthermore, if $h = id$, then 
$$
\displaystyle\inf_{h\in \mathcal H}\max_{u\in Q}I(h(u)) < \displaystyle\max_{u\in Q}I(u) < \min\{c_\sharp,2c_\infty\},
$$ 
in light of Lemma \ref{A7}. Then
$d \in (c_\infty, \min\{c_\sharp,2c_\infty\})$ 
and thus from Lemma \ref{split} the $(Ce)$ condition is satisfied at level $d$. 
Then, by the linking theorem, $d$ is a critical level for $I$.
\qed

%\subsection{Fractional asymptotically linear systems}
%For $s,t,\lambda,\sigma\in (0,1)$, let us consider solutions $(u,v)\in H^s(\R^N)\times H^t(\R^N)$ to 
%\begin{equation*}
%\begin{cases}
%(-\Delta)^s u+u=\displaystyle\frac{u^2+v^2}{1+\sigma (u^2+v^2)}u+\lambda v & \text{in $\R^N$} \\
%(-\Delta)^t v+v=\displaystyle\frac{u^2+v^2}{1+\sigma (u^2+v^2)}v+\lambda u & \text{in $\R^N$}. 
%\end{cases}
%\end{equation*}

\bigskip
\bigskip

\end{document}